\documentclass[preprint,12pt]{elsarticle}

\usepackage{amssymb, amsmath, amsthm}
\usepackage{graphicx}
\usepackage{float, mathtools, subfig, tikz}
\usepackage[colorlinks=true, allcolors=blue]{hyperref}
\usepackage[capitalise, nameinlink]{cleveref}
\newtheorem{thm}{Theorem}[section]

\newtheorem{lem}[thm]{Lemma}
\newtheorem{prop}[thm]{Proposition}

\theoremstyle{definition}
\newtheorem{defn}[thm]{Definition}

\theoremstyle{remark}
\newtheorem{rem}[thm]{Remark}

\crefname{thm}{Theorem}{Theorems}
\crefname{cor}{Corollary}{Corollaries}
\crefname{lem}{Lemma}{Lemmas}
\crefname{prop}{Proposition}{Propositions}
\crefname{defn}{Definition}{Definitions}
\crefname{rem}{Remark}{Remarks}

\newcommand{\R}{\mathbb{R}}
\newcommand{\C}{\mathbb{C}}
\newcommand{\Z}{\mathbb{Z}}
\newcommand{\N}{\mathbb{N}}

\renewcommand{\P}{\mathbb{P}}

\newcommand{\ba}{\ensuremath{\mathbf{a}}}
\newcommand{\BB}{\ensuremath{\mathbf{B}}}
\newcommand{\CC}{\ensuremath{\mathbf{C}}}
\newcommand{\kk}{\ensuremath{\mathbf{k}}}
\newcommand{\nn}{\ensuremath{\mathbf{n}}}
\newcommand{\uu}{\ensuremath{\mathbf{u}}}
\newcommand{\vv}{\ensuremath{\mathbf{v}}}
\newcommand{\xx}{\ensuremath{\mathbf{x}}}
\newcommand{\yy}{\ensuremath{\mathbf{y}}}
\newcommand{\ttheta}{\boldsymbol{\theta}}
\newcommand{\aalpha}{\boldsymbol{\alpha}}
\newcommand{\un}{\boldsymbol{1}}

\journal{European Journal of Combinatorics}

\begin{document}

\begin{frontmatter}

\title{The asymptotic number of zonotopes}

%% Affiliations reproduced from the original source:
%% O. Bodini -> LAGA ; T. Buffière -> LIPN and LAGA (please check).
\author[laga]{Olivier Bodini\corref{cor1}}
\ead{bodini@math.univ-paris13.fr}

\author[lipn,laga]{Théophile Buffière}
\ead{buffiere@math.univ-paris13.fr}
\cortext[cor1]{Corresponding author.}

\affiliation[lipn]{organization={LIPN, Université Sorbonne Paris Nord (P13)},
            postcode={93430},
            city={Villetaneuse},
            country={France}}

\affiliation[laga]{organization={LAGA, Université Sorbonne Paris Nord (P13)},
            postcode={93430},
            city={Villetaneuse},
            country={France}}

\begin{abstract}
We provide a sharp estimate for the asymptotic number of lattice zonotopes inscribed in $[0,n]^d$ as $n$ tends to infinity. Our estimate refines the logarithmic equivalent established by B\'ar\'any, Bureaux, and Lund when the sum of the generators of the zonotope is prescribed. As we shall see, the exponential part of our estimate is a polynomial of degree $d$ in $n^{1/(d+1)}$, plus a term involving the non-trivial zeros of the Riemann zeta function.
Our analysis is based on a correspondence between power sums over primitive vectors (vectors with coprime coordinates) and Eulerian polynomials.
We also analyse some combinatorial properties of lattice zonotopes. In particular, we compute the asymptotic mean and variance of the diameter of the graph of a random lattice zonotope inscribed in $[0,n]^d$ as $n$ goes to infinity, and show that this diameter satisfies a central limit theorem.
\end{abstract}


\end{frontmatter}

%%%%%%%%%%%%%%%%%%%%%%%%%%%%%%%%%%%%%%%%%%%%%%%%%%%

\section{Introduction}\label{introduction}

Enumerating convex lattice polytopes is a long-standing question. In 1980, Arnold bounded the number of convex lattice polygons of area $A$, up to the affine transformations of $\Z^2$, showing that it behaves roughly like $A^{1/3}$ \cite{Arnold:polygons}. In 1992, B\'ar\'any and Pach narrowed those bounds \cite{Barany:numberpolygons}, before B\'ar\'any and Vershik extended the result to higher dimensions \cite{Barany:numberpolytope}. These bounds had in fact been partially improved by Konyagin and Sevast'yanov as early as 1984 \cite{Konyakin:bound_polyhedron}.
Vershik later asked for the number of lattice polytopes in a box \cite{Vershik:limit}; for instance, in dimension 2, the number of convex lattice polygons inscribed in the square $[-n, n]^2$.

A natural intermediate step is to enumerate the number $p(n)$ of convex polygonal lines from $(0,0)$ to $(n,n)$ on the lattice $\Z^2$ (referred to as lattice chains in the sequel). An asymptotic equivalent of its logarithm was found independently by B\'ar\'any \cite{Barany:limitshape}, Vershik \cite{Vershik:limit}, and Sinai \cite{Sinai:polygonallines}, who showed that
\begin{align}\label{kappa}
    p(n) = \exp \left(3 \kappa^{1/3} n^{2/3} (1 + o(1))\right), \qquad \text{where} \quad \kappa = \frac{\zeta(3)}{\zeta(2)} .
\end{align}

More recently, Bodini, Duchon, Jacquot, and Mutafchiev \cite{Bodini:polyomino} and Bureaux and Enriquez \cite{Bureaux:polygons} refined this result by computing an asymptotic equivalent of the number of convex lattice chains. The former used analytic combinatorics tools similar to those of the present paper, starting from the study of convex polyominoes. The latter started from Sinai's model and used the local limit theorem of Bogachev and Zarbaliev \cite{Bogatchev:limit}. Both obtained
$$p(n) \sim \frac{e^{- 2 \zeta'(-1)}}{(2\pi)^{7/6} \sqrt{3}\, \kappa^{1/18} n^{17/18}} \exp \left( 3 \kappa^{1/3} n^{2/3} + I\left(\left(\frac{\kappa}{n}\right)^{1/3}\right) \right), $$
where $I$ is a sum over the non-trivial zeros of the Riemann zeta function; under the Riemann hypothesis, $I\left(\left(\kappa/n\right)^{1/3}\right) = O(n^{1/6})$.

However, computing the number of lattice polytopes remains an open question, since going from convex lattice chains to convex polygons requires a delicate analysis of how the chains are connected.

This paper deals with the enumeration of zonotopes in $\R^d$, for any dimension $d$.
A \emph{zonotope} is a convex body defined as the Minkowski sum of $k$ segments, called its \emph{generators}. It is in one-to-one correspondence with the convex polygonal line made up of those generators. More precisely, a \emph{lattice zonotope} $Z$ (or integral zonotope) is a polytope for which there exist $k \in \N$ and $\vv_1, \dots, \vv_k \in \Z^d$ such that, up to translation,
$$ Z = \left\{ \sum_{i=1}^k \alpha_i \vv_i  \;\middle|\; \alpha_1,\dots,\alpha_k \in [0,1]\right\}. $$

Being Minkowski sums of segments, zonotopes are as much combinatorial objects as geometric ones, which makes the question of their enumeration tractable. They are also important objects in their own right: they are used for solving systems of polynomial equations \cite{Huber:zonotopes} and as bounding volumes in spatial representations \cite{Guibas:zonotopes}. From a theoretical perspective, they appear in the theory of the simplex method \cite{Delpia:simplex} and in the study of polytope diameters. Evaluating their number is one more step towards understanding these objects and their relation to convex polytopes.

In 2018, B\'ar\'any, Bureaux, and Lund \cite[Theorem 1.2]{Barany:zonotopes} computed the logarithmic asymptotics of the number of lattice zonotopes contained in a cone $C$ and ending at a prescribed point $n\kk$. In the case of the zonotopes generated by vectors of $\Z_+^d$ and ending at $n\un$, this gives
$$ z'_d(\R_+^d, n \un) \underset{n \rightarrow + \infty}{ = } \exp \left( (d+1) \left(\frac{\zeta(d+1)}{\zeta(d)}\right)^{\frac{1}{d+1}}  n^{\frac{d}{d+1}} (1 + o(1)) \right).$$

The present paper refines the results of \cite{Barany:zonotopes} and \cite{Bureaux:polygons}. Throughout the paper, we set
$$\kappa_d = \frac{2^{d-1}\zeta(d+1)}{\zeta(d)}.$$
Our main result is the following.

\begin{thm}\label{theorem_1}
Let $d\in \N$, $d \geq 2$, and denote by $z_d(n)$ the number of lattice zonotopes inscribed in $[0,n]^d$ (see Section~\ref{section_combi}). There exist real numbers $\alpha_d$, $\beta_d$, $\gamma_d$, a polynomial $\widehat{Q}_d$ of degree $d$ without constant term, and a quantity $I_{\text{crit}}(d, n)$ depending on the non-trivial zeros of the Riemann zeta function (all of which are explicit, or explicitly computable, see Theorem~\ref{theorem_1_complet}), such that, as $n$ grows large,
\begin{align}
     z_d(n) \sim  \alpha_d e^{\gamma_d} n ^{\beta_d} \exp \left( \widehat{Q}_d(n^\frac{1}{d+1}) + I_{\text{crit}}(d, n) \right).
\end{align}
\end{thm}

We establish this asymptotic estimate in every dimension, highlighting in the process a general correspondence between sums over coprime numbers and Eulerian polynomials. Our work essentially deals with generating functions, which are ubiquitous in analytic combinatorics (see for instance Flajolet and Sedgewick \cite{Flajolet:AC}). This result calls for two remarks.

\begin{rem}
Unexpectedly, the exponential part is in general not a monomial, as it is in dimensions 2 and 3 (see below), but a polynomial in the $(d+1)$-th root of the size of the box.
\end{rem}

\begin{rem}
As already observed in dimension 2, the estimate contains a term depending on the non-trivial zeros of the $\zeta$ function. In Section~\ref{section_numerical_discu}, we discuss the value of this term, and extend the observation of \cite{Bodini:polyomino, Bureaux:polygons} that $I_{\text{crit}}$ is negligible in dimension 2 for any ``computably'' large $n$ (typically $n < 10^{20}$).
\end{rem}

To illustrate this theorem, we give the following table.

\begin{table}[H]
    \centering
    \small
    \setlength{\tabcolsep}{4pt}
    \begin{tabular}{|c|c|c|c|}
        \hline
         & $\alpha_d e^{\gamma_d}$ & $\beta_d$ & $\widehat{Q}_d(X)$\\ [0.5ex] 
        \hline
        $d=2$ & $\frac{2^{1/9}  3^{13/18}  \zeta(3)^{2/9} e^{- 4 \zeta'(-1)} }{6 \pi^{16/9}}$  &  $- \frac{11}{9}$ & $\frac{2^{2/3}   3^{4/3}  \zeta(3)^{1/3} }{ \pi^{2/3}  }X^2 $  \\ [0.5ex] 
        \hline
        $d=3$ &  $\frac{2^{5/8}  3^{3/4}  5^{7/8} e^{- 4 \zeta'(-2)} }{120\, \zeta(3)^{1/8} \pi}$ & $- \frac{13}{8}$ & $ \frac{2^{9/4}  \pi}{3^{1/2} 5^{1/4} \zeta(3)^{1/4}  }X^3$\\ [0.7ex] 
        \hline
        $d=4$ & $\left(\frac{3^{284} \zeta(5)^{142} }{2^{522}\, 5^{83}\,  \pi^{1658}  } \right)^{\frac{1}{450}}  e^{- \frac{8}{3}\zeta'(-3) -\frac{16}{3} \zeta'(-1) - \frac{8\zeta(3)^2}{225\zeta(5)}}$ & $- \frac{521}{225}$ & $\frac{5 \cdot 720^{1/5} \zeta(5)^{1/5} }{\pi^{4/5}  } X^4 +  \frac{16\, \zeta(3) }{720^{2/5}\, \zeta(5)^{2/5}\, \pi^{2/5}  }X^2$\\
        \hline
    \end{tabular}
    \caption{Parameters of the asymptotic equivalent of the number of zonotopes inscribed in a hypercube of dimension 2, 3, and 4.}
    \label{tab:number_of_zonotopes}
\end{table}

The quantities $I_{\text{crit}}(2, n)$, $I_{\text{crit}}(3, n)$, and $I_{\text{crit}}(4, n)$ are numerically approximated in (\ref{numeric_I_crit}).

\subsection{Moments of the diameter of lattice zonotopes}
Once the combinatorial description of zonotopes and Theorem~\ref{theorem_1} have been established, one can compute moments of combinatorial parameters of zonotopes. In this paper, we are interested in two parameters: the diameter of the graph of a zonotope and the number of occurrences of a given generator in a random zonotope. \medskip

The diameter (of the graph) of a polytope of a given size is a key combinatorial parameter about which little is known. Besides its own scientific interest (see for instance the Hirsch conjecture), it is connected to the complexity of the simplex algorithm, which makes upper bounds on this quantity highly sought after (see \cite{Delpia:simplex} for more references). In this perspective, the largest possible diameter of a lattice zonotope contained in $[0,n]^d$ is a natural quantity to consider. \medskip

We compute asymptotic estimates of the mean and the variance of the diameter of a random lattice zonotope inscribed in $[0,n]^d$.

\begin{thm}\label{diameter_zono}
Let $\mu_{\text{diam}}^n$ and $(\sigma^2)^n_{\text{diam}}$ denote the mean and the variance of the diameter of a lattice zonotope drawn uniformly at random among the lattice zonotopes inscribed in $[0,n]^d$. Then
$$ \mu_{\text{diam}}^n \underset{n \rightarrow +\infty}{\sim} \frac{\kappa_d^{1/(d+1)}}{\zeta(d+1)} \, n^{\frac{d}{d+1}}, \qquad (\sigma^2)^n_{\text{diam}} \underset{n \rightarrow +\infty}{\sim} \left(1 - \frac{1}{2^{d}} - \frac{d}{(d+1)\,\zeta(d+1)}\right)\mu_{\text{diam}}^n.$$
\end{thm}

To summarize, this article is organized as follows. We begin with establishing the generating function associated with the combinatorial class of zonotopes (Section~\ref{section_combi}). Then, Section~\ref{section_eulerset} is dedicated to matching power sums over primitive vectors with Eulerian polynomials, which makes it possible to carry out the computations in any dimension. In Section~\ref{section_analysis}, we give an integral formula for the generating function and prove technical lemmas about it. Theorem~\ref{theorem_1} is established in Section~\ref{section_proof_main} using a multidimensional saddle-point method. Finally, the last two sections are respectively dedicated to discussing the asymptotic equivalent obtained, and to computing the first moments of the number of pairwise non-collinear generators (that is, the diameter) and of the number of occurrences of a given primitive generator.

%%%%%%%%%%%%%%%%%%%%%%%%%%%%%%%%%%%%%%%%%%%%%%%%%%%

\section{The class of zonotopes: a combinatorial approach}\label{section_combi}

Recall from the introduction that a \emph{lattice zonotope} $Z$ (or integral zonotope) is a polytope for which there exist $k \in \N$ and $\vv_1, \dots, \vv_k \in \Z^d$ such that, up to translation,
$$ Z = \left\{ \sum_{i=1}^k \alpha_i \vv_i  \;\middle|\; \alpha_1,\dots,\alpha_k \in [0,1]\right\}. $$
Since we study zonotopes inscribed in a hypercube, we consider lattice zonotopes up to translation in the sequel. In other words, all the generating segments start at $\boldsymbol{0}$.

\begin{figure}[H]
    \centering
    \subfloat[\centering 2 dimensions]{{\includegraphics[width=6cm]{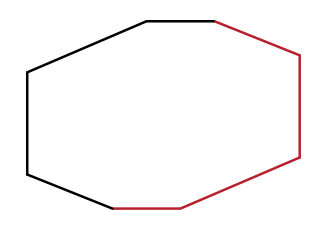} }}%
    \qquad
    \subfloat[\centering 3 dimensions]{{\includegraphics[width=6cm]{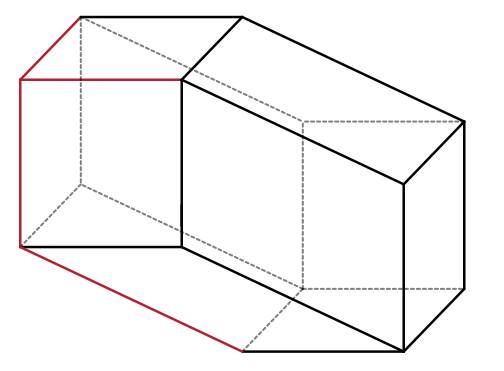} }}%
    \caption{Zonotopes generated by 4 vectors (in red).}%
    \label{fig:example}%
\end{figure}

A family of generators $\vv_1,\dots,\vv_k \in \Z^d\setminus\{\boldsymbol{0}\}$ uniquely defines a zonotope, but the converse is not true: for instance, two collinear generators $\vv$ and $2\vv$ give the same zonotope as the single generator $3\vv$. We denote by $\textsc{Zon}_d$ the set of $d$-dimensional lattice zonotopes up to translation. The projection of $Z$ on the $j$-th coordinate axis is a segment of length $\sum_{i=1}^k |v_{i,j}|$, so that, up to translation, the bounding box of $Z$ is $[0,n_1]\times\dots\times[0,n_d]$ with $n_j = \sum_{i=1}^k |v_{i,j}|$. We say that $Z$ is \emph{inscribed} in this box. We denote by $z_d(n)$ the number of lattice zonotopes inscribed in $[0,n]^d$. Note that this count includes the zonotopes of dimension less than $d$; for instance, the segment $[\boldsymbol{0}, n\un]$ is inscribed in $[0,n]^d$. These degenerate zonotopes, all of whose generators lie in a hyperplane, form a vanishing proportion of all zonotopes (see \cite[Proposition 5.2]{Barany:zonotopes} for the analogous statement for zonotopes in a cone).
As was previously done in the study of polygonal chains \cite{Bureaux:polygons}, we state a bijective correspondence between zonotopes and finitely supported functions encoding their generators. \medskip

A vector $\vv \in \Z^d \setminus\{\boldsymbol{0}\}$ is called primitive if the gcd of its non-zero coordinates is 1, namely if the intersection of the segment $[\boldsymbol{0},\vv]$ with $\Z^d$ consists of its two endpoints only. Let $\P_d$, respectively $\P_{d+}$, denote the set of primitive vectors of $\Z^d$, respectively $(\Z_+)^d$, where $\Z_+ = \{0,1,2,\dots\}$. Then let $\P_d^\star$ denote the set of primitive vectors whose first non-zero coordinate is positive, following the convention of \cite{Deza:primitive}. The set $\Omega$ of finitely supported functions $\omega : \P_d^\star \rightarrow \N$ is in one-to-one correspondence with $\textsc{Zon}_d$. More precisely:
\begin{itemize}
    \item the function $\omega$ associated with a zonotope $Z$ is defined, for all $\vv \in \P_d^\star$, as the number of occurrences of $\vv$ in an edge of $Z$ collinear to $\vv$ if there is one (that is, the length of this edge divided by the length of $\vv$), and 0 otherwise;
    \item given a function $\omega$, the list of generators $\{ \omega(\vv) \vv, \; \vv \in \text{supp}(\omega)\}$ defines $Z$. In particular, $\sum_{\vv\in  \P_d^\star}\omega(\vv)\vv$ is the endpoint of the zonotope.
\end{itemize}

This combinatorial structure is a set of elements with repetitions allowed and without order. It is known as a multiset in the framework of the symbolic method (see \cite[Chapter~I, pp.~26--27]{Flajolet:AC}). This allows us to apply well-known analytic tools to the generating function of the multisets of elements of $\P_d^\star$, whose study is part of the contribution of this paper. \medskip

Let $Zon_d$ denote the multivariate generating function of the class $\textsc{Zon}_d$, that is
$$ Zon_d(\xx) = \sum_{\nn \in \N^d} p_{\nn}\, x_1 ^{n_1}\cdots x_d^{n_d}, $$
where $p_{\nn}$ is the number of zonotopes inscribed in the box $[0,n_1] \times \dots \times [0,n_d]$. The variable $x_i$ thus encodes the width of zonotopes in the $i^{\text{th}}$ direction. As usual, we write $\xx^{\nn} =x_1 ^{n_1}\cdots x_d^{n_d}$, $\un$ for the vector $(1,1, \dots,1)$, and $\xx \cdot \yy$ for the canonical scalar product of $\xx$ and $\yy$. We also denote by $[\vv]$ the vector of the absolute values of the coordinates of $\vv$, that is $(|v_1|,\dots, |v_d|)$. \medskip

Adding a generator $\omega(\vv)\vv$ to a zonotope increases the width of its bounding box in the $i^{\text{th}}$ direction by $\omega(\vv)|v_i|$, for $1 \leq i \leq d$. Hence we can rewrite the generating function using finitely supported functions as follows:
$$Zon_d(\xx) =  \sum_{\omega \in \Omega} \prod_{\vv \in \P_d^\star} \xx^{\omega(\vv) [\vv]}, $$
and a factorization finally gives the usual form of the generating function of a multiset:
\begin{align}\label{partition_function}
 Zon_d(\xx) =  \prod_{ \vv \in \P_d^\star}\left(1 - \xx^{[\vv]} \right)^{-1}.   
\end{align}

This function is analytic on the open unit polydisk $\{\xx \in \C^d, \; |x_i| < 1 \text{ for } 1 \leq i \leq d\}$, and its dominant singularity is at $\un$. In the sequel, we use the change of variables $x_i = e^{-\theta_i}$ and study the function as $\ttheta$ goes to $\boldsymbol{0}$. This expression is the partition function of the Boltzmann model considered in \cite{Bureaux:polygons}:
$$ \prod_{ \vv \in \P_d^\star}\left(1 - e^{ - \ttheta \cdot [\vv]} \right)^{-1}. $$

Denote by $d(\vv)$ the number of non-zero coordinates of $\vv$. For a given $\uu \in \P_{d+}$, there are exactly $2^{d(\uu)-1}$ vectors $\vv \in \P_d^\star$ with $[\vv] = \uu$ (the first non-zero coordinate being positive). Grouping these vectors, we get
\begin{align}\label{generating_function}
    Zon_d\left(e^{- \ttheta}\right) =  \prod_{ \vv\in \P_{d+}} \left( 1 - e^{-  \ttheta \cdot \vv } \right)^{- 2^{d(\vv)-1}}.
\end{align}

%%%%%%%%%%%%%%%%%%%%%%%%%%%%%%%%%%%%%%%%%%%%%%%%%%%

\section{Integral representation of Eulerian sums over coprime vectors}\label{section_eulerset}

The proof of Theorem~\ref{theorem_1}, as a generalization of the result of \cite{Bureaux:polygons}, follows the same framework: we use an integral representation of the logarithm of the generating function given by the Mellin transform. This leads to an estimate of the generating function near $\ttheta = \boldsymbol{0}$ by successively shifting the integration contour of the inverse transform to the left. It turns out that this method can be carried out in any dimension thanks to a link between Eulerian polynomials and power sums over primitive vectors. \medskip 

Let $\aalpha \in \N^d$ and $t > 0$, and denote by $\text{int}(\R^d_{+})$ the interior of $\R^d_{+}$. The aim of this section is to study the sum
$$S_{\aalpha} (e^{-t}) = \sum_{\vv \in \P_{d+} \cap\, \text{int}(\R^d_{+})} \vv^{\aalpha} e^{ - t\, \un \cdot \vv} = \sum_{n = 1}^{+ \infty} \Bigg( \sum_{\substack{\vv \in \P_{d+}\cap\, \text{int}(\R^d_{+}) \\ \un \cdot \vv = n}} \vv^{\aalpha} \Bigg) e^{ - t n}. $$
The inner sums, over the primitive vectors with positive coordinates and $L^1$ norm equal to $n$, are denoted by $S_n(\aalpha)$. When the dimension needs to be specified, we write $S^{(d)}_{\aalpha}$ and $S^{(d)}_n(\aalpha)$. \medskip

\begin{rem}
The sum only involves vectors whose coordinates are all positive. Equivalently, it can be seen as a sum over all elements of $\P_{d+}$ with the convention that $\vv^{\aalpha} = 0$ whenever $\vv$ has a zero coordinate (even if the corresponding coordinate of $\aalpha$ is zero).
\end{rem}

Before stating the main proposition of this section, we recall the definition of the Mellin transform. Let $g : (0,+\infty) \to \C$ be a locally integrable function such that $g(t) = O(t^{-a})$ as $t \rightarrow 0^+$ and $g(t) = O(t^{-b})$ as $t \rightarrow +\infty$, with $a<b$. For all $s$ with $a< \Re(s) <b$, the Mellin transform of $g$ is defined as
$$ \mathcal{M}\left[g(t)\right](s) = \int_0^{+\infty} g(t)\,t^{s-1}\,dt.$$
For instance, the Mellin transform of $e^{-t}$ is the Gamma function: $\mathcal{M}[e^{-t}](s) = \Gamma(s)$ for $\Re(s) > 0$.

\begin{prop}\label{proposition_eulerian}
For all $s \in \C$ with $\Re(s) > \un \cdot \aalpha + d$,
 \begin{align}
     \mathcal{M}\left[ S_{\aalpha}(e^{-t}) \right](s) = \frac{1}{\zeta(s - \aalpha \cdot \un )} \mathcal{M}\left[ \frac { \prod_{i = 1}^{d}  A_{\alpha_i}(e^{-t})}{(1-e^{-t}) ^{\aalpha \cdot \un + d}}  \right](s) ,
 \end{align}
 where $A_{k}$ denotes the $k$-th Eulerian polynomial (with the normalization given below). In particular, $\mathcal{M}\left[ S_{\aalpha}(e^{-t}) \right]$ has a meromorphic continuation to $\C$.
\end{prop}

\medskip

Eulerian polynomials were first studied by Euler \cite{Euler}, in connection with the series expansion of $\frac{y e^x}{1 - y e^x}$. Denote by $A(n,m)$ the Eulerian numbers, that is, the number of permutations of $\{1,\dots,n\}$ with exactly $m$ descents. In this paper, we use the following normalization of Eulerian polynomials:
$$ A_0(y) = y \qquad \text{and} \qquad A_n(y) = \sum_{m= 0}^{n-1} A(n,m)\,y^{n-m} \quad \text{for } n \geq 1,$$
which differs from the most common one by a factor $y$. Hence the first Eulerian polynomials are $A_1(y) = y$, $A_2(y) = y^2 + y$, and $A_3(y) = y^3 + 4 y^2 + y$. With this normalization, they satisfy, for $|y|<1$ and $n \geq 0$,
\begin{align}\label{eulerien_poly}
    \sum_{k=1}^\infty k^n y^{k}  &= \frac{ A_n (y)}{(1-y)^{n+1}}.
\end{align}

\begin{lem}\label{lem_eulerian_sum_notation} 
For any complex number $y$ with $|y| < 1$, 
\begin{align}\label{relation_lem_eulerian}
    \sum_{k \geq 1} k^{\un \cdot \aalpha} S_{\aalpha} (y^k) =  \frac { \prod_{i = 1}^{d}  A_{\alpha_i}(y)}{(1-y) ^{\aalpha \cdot \un + d}} .
\end{align} 
\end{lem}

\begin{proof}
This relation relies on the definition of Eulerian polynomials and on the partition
$$\Z_{>0}^d =\bigsqcup_{k \geq 1} \left\{ k\vv, \; \vv \in \P_{d+} \cap\, \text{int}(\R^d_{+}) \right\}.$$ The left-hand side can be written as
$$ \sum_{k \geq 1 } k^{\un \cdot \aalpha}  S_{\aalpha} (y^k) = \sum_{k \geq 1} \sum_{n \geq 1} \sum_{\substack{ \vv \in \P_{d+} \cap\, \text{int}(\R^d_{+}) \\   \un \cdot \vv = n }}(k \vv)^{\aalpha}\, y^{k\, \un \cdot \vv}.$$
The set of summation is the partition mentioned above. For $|y| < 1$, the series $\sum_{\vv} \vv^{\aalpha} |y|^{\un \cdot \vv}$ over $\vv \in \Z_{>0}^d$ converges; hence the Fubini--Tonelli theorem allows us to rearrange the sums, and relation (\ref{eulerien_poly}) concludes the proof:
$$ \sum_{\vv \in \Z_{>0}^d} \vv ^{\aalpha} y^{\un \cdot \vv}   = \prod_{i = 1}^{d} \left( \sum_{v_i \geq 1} v_i^{\alpha_i} y^{v_i} \right)  = \frac { \prod_{i = 1}^{d}  A_{\alpha_i}(y)}{(1-y) ^{\aalpha \cdot \un + d}}.  $$
\end{proof}

With our normalization, $A_n$ has degree $n$ for $n \geq 1$, $A_0$ has degree 1, and $A_n$ has no constant term. Hence the degree of $\prod_{i = 1}^{d}  A_{\alpha_i}(y)$ is at most $\aalpha \cdot \un + d$, with equality for $\aalpha = \textbf{0}$. The vectors $\aalpha$ needed in this paper are $\textbf{0}$, $(1,0,\dots,0)$, and $(2,0,\dots,0)$, for which the product of polynomials respectively equals $y^{d}$, $y^{d}$, and $y^{d+1}+ y^{d}$.

Note that $S_n(\aalpha)$ is a power sum over the primitive points with positive coordinates lying on the dilate $n\Delta_{d-1}$ of the $(d-1)$-dimensional standard simplex.

\begin{proof}[Proof of Proposition~\ref{proposition_eulerian}]
Since
$$0 \leq S_{\aalpha}(e^{-t}) \leq \frac{\prod_{i} A_{\alpha_i}(e^{-t})}{(1-e^{-t})^{\aalpha\cdot\un+d}},$$
we have $S_{\aalpha}(e^{-t}) = O(t^{-\aalpha \cdot \un - d})$ as $t \rightarrow 0^+$, and $S_{\aalpha}(e^{-t}) = O(e^{-dt})$ as $t \rightarrow +\infty$. Hence the Mellin transform of $S_{\aalpha}(e^{-t})$ is defined for $\Re(s) > \aalpha \cdot \un + d$, and so is the Mellin transform of the right-hand side of (\ref{relation_lem_eulerian}) with $y = e^{-t}$. For such $s$, the series $\sum_{k \geq 1} k^{\aalpha \cdot \un - \Re(s)}$ converges, so that, by the Fubini--Tonelli theorem and the change of variable $t' = kt$,
$$ \int_0^{+\infty} \sum_{k \geq 1 } k^{\aalpha \cdot \un} S_{\aalpha}(e^{-kt})\, t^{s-1}\, d t= \zeta(s-\aalpha\cdot \un) \: \mathcal{M}\left[ S_{\aalpha}(e^{-t}) \right](s).$$
Replacing the left-hand side using Lemma~\ref{lem_eulerian_sum_notation}, and dividing by $\zeta(s-\aalpha\cdot \un)$, which does not vanish for $\Re(s) > \aalpha\cdot\un + 1$, gives the proposition. Finally, expanding $\prod_{i} A_{\alpha_i}(y)/(1-y)^{\aalpha\cdot\un+d}$ as $\sum_{n \geq 1} R(n) y^n$, where $R$ is a polynomial, shows that the Mellin transform of the right-hand side is $\Gamma(s) \sum_{n\geq 1} R(n) n^{-s}$, a finite linear combination of functions $\Gamma(s)\zeta(s-j)$, hence meromorphic on $\C$.
\end{proof}

The Mellin transform is widely used to estimate asymptotic behaviors through the correspondence between the asymptotic expansion of a function and the poles of its transform, as we do in the next section (and as is done in \cite{Bodini:polyomino} and \cite{Bureaux:polygons}). Proposition~\ref{proposition_eulerian} makes it possible to compute an asymptotic expansion of any of the sums $S_{\aalpha}(y)$ near $y =1$. 

%%%%%%%%%%%%%%%%%%%%%%%%%%%%%%%%%%%%%%%%%%%%%%%%%%%

\section{Asymptotic analysis of the generating function}\label{section_analysis}

We start the analysis of the generating function by giving an integral formula for its logarithm, and an asymptotic equivalent of its partial derivatives. The second part focuses on the univariate function $\theta \mapsto Zon_d(e^{-\theta}\un)$, for which we compute an asymptotic expansion. Both parts are needed for the saddle-point analysis conducted in the next section.

\subsection{Integral formula and partial derivatives of the generating function}

\begin{lem}\label{lemma_integral_form_gg}
Let $c>d$. For all $\ttheta \in (0, + \infty)^d$, we have
\begin{equation}\label{equa_integral}
  \log\left(  Zon_d(e^{-\ttheta}) \right) =   \frac{1}{2 i \pi} \int _{c - i\infty} ^{ c+i \infty} \frac{\zeta(s+1)\Gamma(s)}{\zeta(s)}  \sum_{ \vv\in \Z^d_+ \setminus \{\boldsymbol{0}\}} \frac{2^{d(\vv)-1}}{( \ttheta \cdot \vv)^s}  ds.  
\end{equation}
\end{lem}

\begin{proof}
We first take the logarithm of $Zon_d(e^{-\ttheta})$ as written in (\ref{generating_function}), where $d(\vv)$ is the number of non-zero coordinates of $\vv$, and expand the logarithm in Taylor series:
$$ \log\left(  Zon_d(e^{-\ttheta}) \right) = - \sum_{ \vv\in \P_{d+}} 2^{d(\vv)-1} \log \left( 1 - e^{-  \ttheta \cdot \vv } \right)  = \sum_{ \vv\in \P_{d+}} \sum_{m \geq 1} \frac{ 2^{d(\vv)-1} }{m} e^{-m\,  \ttheta \cdot \vv }.   $$
For all real numbers $z > 0$ and $c>0$, the Mellin inversion formula gives
$$e^{-z} =  \frac{1}{2 i \pi} \int _{c - i\infty} ^{ c+i \infty} \Gamma(s) z^{-s} ds.$$
We take $c > d$, which ensures that $\sum_{\vv \in  \P_{d+} } (\ttheta \cdot \vv )^{-c} < + \infty$. Since moreover $\sum_{ m\geq1} m^{-1-c} < + \infty$ and $\int_{c+i\R} |\Gamma(s)|\,|ds| < +\infty$, the Fubini--Tonelli theorem leads to
\begin{align*}
\log\left(  Zon_d(e^{-\ttheta}) \right) &=   \frac{1}{2 i \pi} \int _{c - i\infty} ^{ c+i \infty} \sum_{ \vv\in \P_{d+}} \sum_{m \geq 1} \frac{ 2^{d(\vv)-1} }{m} \frac{\Gamma(s)}{(m\, \ttheta \cdot \vv)^s}  ds \\
&=  \frac{1}{2 i \pi} \int _{c - i\infty} ^{ c+i \infty} \zeta(s+1)\Gamma(s) \sum_{ \vv\in \P_{d+}} \frac{ 2^{d(\vv)-1} }{(\ttheta \cdot \vv)^s}  ds.
\end{align*}
Finally, the sum over primitive vectors can be completed into a sum over $\Z^d_+ \setminus \{\boldsymbol{0}\}$ using the partition $\Z^d_+ \setminus \{\boldsymbol{0}\} = \bigsqcup_{k \geq 1} k \P_{d+}$ and the fact that $d(k\vv) = d(\vv)$:
$$ \sum_{ \vv\in \Z^d_+\setminus \{\boldsymbol{0}\}} \frac{2^{d(\vv)-1}}{( \ttheta \cdot \vv)^s} = \sum_{k \geq 1 } \frac{1}{k^s} \sum_{ \vv\in \P_{d+}} \frac{2^{d(\vv)-1}}{( \ttheta \cdot \vv)^s} = \zeta(s) \sum_{ \vv\in \P_{d+}} \frac{2^{d(\vv)-1}}{( \ttheta \cdot \vv)^s}, $$
which gives (\ref{equa_integral}).
\end{proof}

This general formula allows us to derive the following lemma on the asymptotic behavior of the partial derivatives of $\log\left(  Zon_d(e^{-\ttheta}) \right)$. 

\begin{lem}\label{lemma_derivativ}
Let $\varepsilon \in (0,1)$. For all $\kk = (k_1,\dots,k_d) \in \N^d$, as $\ttheta \rightarrow \boldsymbol{0}$ with $\varepsilon \leq \theta_i/\theta_j \leq 1/\varepsilon$ for all $i,j$,
$$\frac{\partial^{k_1 + k_2 +\dots+ k_d}}{\partial\theta_1^{k_1} \partial\theta_2^{k_2}\cdots\partial \theta_d^{k_d}} \log\left(  Zon_d(e^{-\ttheta}) \right) \sim (-1)^{k_1 + \dots + k_d}\, \frac{2^{d-1}\zeta(d+1)}{\zeta(d)}\,\frac{\prod_{i=1}^d k_i !}{\theta_1^{k_1+1}\cdots\theta_d^{k_d+1}}.$$
\end{lem}

The sum appearing in (\ref{equa_integral}) is a linear combination of Barnes zeta functions $\sum_{\nn \in \Z_+^\delta}(w + \ttheta\cdot\nn)^{-s}$ (see \cite{Barnes:multiple, Ruijsenaars:Barnes}). We do not need the general theory of these functions: the only properties we use are contained in the following elementary lemma, which deals with the product $\Gamma(s)F_d(s,\ttheta)$, where
$$F_d(s,\ttheta) = \sum_{ \vv\in \Z^d_+ \setminus \{\boldsymbol{0}\}} \frac{2^{d(\vv)-1}}{( \ttheta \cdot \vv)^s}, \qquad \Re(s)>d.$$

\begin{lem}\label{lemma_Barnes}
Let $\varepsilon \in (0,1)$, $\eta \in (0,\frac{\pi}{2})$, $\delta \in (0,1)$, and $\kk \in \N^d$. For $\ttheta \in (0,+\infty)^d$, the function $s \mapsto \frac{\partial^{\kk\cdot\un}}{\partial\ttheta^{\kk}}\left(\Gamma(s) F_d(s,\ttheta)\right)$ has a meromorphic continuation to $\Re(s) > d-1$, whose only pole is a simple pole at $s = d$ with residue
$$A_{\kk}(\ttheta) = \frac{\partial^{\kk\cdot\un}}{\partial\ttheta^{\kk}} \left(\frac{2^{d-1}}{\theta_1\cdots\theta_d}\right) = (-1)^{\kk\cdot\un}\, 2^{d-1} \frac{\prod_{i=1}^d k_i!}{\theta_1^{k_1+1}\cdots\theta_d^{k_d+1}}.$$
Moreover, there exists a constant $C>0$ such that, for all $\ttheta \in (0,+\infty)^d$ with $\theta_i \geq \varepsilon|\ttheta|$ for $1 \leq i \leq d$, and for all $s$ with $\Re(s) \in [d-1+\delta, d+1]$,
$$ \left| \frac{\partial^{\kk\cdot\un}}{\partial\ttheta^{\kk}}\left(\Gamma(s) F_d(s,\ttheta)\right) \right| \leq C\, e^{-\eta|\Im(s)|}\, \left(1 + \frac{1}{|s-d|}\right)|\ttheta|^{-\kk\cdot\un - \Re(s)}.$$
\end{lem}

\begin{proof}
For $\boldsymbol{x} \in \C^d$ with $\Re(x_j) > 0$, set
$$H(\boldsymbol{x}) = \sum_{ \vv\in \Z^d_+ \setminus \{\boldsymbol{0}\}} 2^{d(\vv)-1} e^{-\boldsymbol{x}\cdot\vv} = \sum_{\emptyset \neq J \subseteq \{1,\dots,d\}} 2^{|J|-1} \prod_{j \in J} \frac{1}{e^{x_j}-1}.$$
For $\Re(s) > d$, we have $\Gamma(s)F_d(s,\ttheta) = \int_0^{\infty} t^{s-1} H(t\ttheta)\, dt$, and differentiating under the integral sign,
$$\frac{\partial^{\kk\cdot\un}}{\partial\ttheta^{\kk}}\left(\Gamma(s) F_d(s,\ttheta)\right) = \int_0^{\infty} t^{s-1} H_{\kk}(t,\ttheta)\, dt, \qquad H_{\kk}(t,\ttheta) = \frac{\partial^{\kk\cdot\un}}{\partial\ttheta^{\kk}} H(t\ttheta).$$
By homogeneity, $H_{\kk}(t, \lambda\ttheta) = \lambda^{-\kk\cdot\un}H_{\kk}(\lambda t,\ttheta)$, so that the left-hand side is multiplied by $\lambda^{-\kk\cdot\un - s}$ when $\ttheta$ is replaced by $\lambda\ttheta$; it is therefore enough to consider $\ttheta$ in the compact set $K_\varepsilon = \{|\ttheta| = 1, \; \theta_i \geq \varepsilon\}$.

Let $S_\eta = \{t \in \C\setminus\{0\}, \; |\arg t| \leq \eta\}$. For $\ttheta \in K_\varepsilon$, the function $t \mapsto H(t\ttheta)$ is holomorphic on a neighbourhood of $S_\eta$, since the poles of $1/(e^{t\theta_j}-1)$ lie on the imaginary axis. Moreover, writing $\frac{1}{e^x-1} = \frac{1}{x}\varphi(x)$ with $\varphi(x) = \frac{x}{e^x-1}$ holomorphic on $|x|<2\pi$ and $\varphi(0)=1$, we get, for $t \in S_\eta$ with $|t| \leq 1$,
$$H(t\ttheta) = \frac{2^{d-1}}{t^d\,\theta_1\cdots\theta_d} + R(t,\ttheta), \qquad |R(t,\ttheta)| \leq C_1 |t|^{-(d-1)},$$
since $\prod_j \varphi(t\theta_j) = 1 + O(|t|)$ and the terms with $|J| \leq d-1$ are $O(|t|^{-(d-1)})$; and, for $t \in S_\eta$ with $|t| \geq 1$, $|H(t\ttheta)| \leq C_1 e^{-\varepsilon \cos(\eta)|t|}$. All these functions are holomorphic in $\ttheta$ on a complex neighbourhood of $K_\varepsilon$, and the bounds hold uniformly there; by Cauchy's estimates, the same bounds hold for the derivatives with respect to $\ttheta$, namely $H_{\kk}(t,\ttheta) = A_{\kk}(\ttheta)t^{-d} + R_{\kk}(t,\ttheta)$ with $|R_{\kk}| \leq C_2|t|^{-(d-1)}$ for $|t| \leq 1$, and $|H_{\kk}(t,\ttheta)| \leq C_2 e^{-\varepsilon\cos(\eta)|t|}$ for $|t|\geq 1$, uniformly for $\ttheta \in K_\varepsilon$ and $t \in S_\eta$.

Let $\phi \in [-\eta,\eta]$. By Cauchy's theorem and the above bounds, the integral over $(0,+\infty)$ can be replaced, for $\Re(s)>d$, by the integral over the half-line $e^{i\phi}(0,+\infty)$:
\begin{align*}
\int_0^{\infty} t^{s-1} H_{\kk}(t,\ttheta)\, dt &= e^{i\phi s}\int_0^{\infty} r^{s-1} H_{\kk}(re^{i\phi},\ttheta)\, dr \\
&= \frac{e^{i\phi(s-d)} A_{\kk}(\ttheta)}{s-d} + e^{i\phi s}\int_0^{1} r^{s-1} R_{\kk}(re^{i\phi},\ttheta)\, dr \\
&\qquad + e^{i\phi s}\int_1^{\infty} r^{s-1} H_{\kk}(re^{i\phi},\ttheta)\, dr.
\end{align*}
The right-hand side is meromorphic on $\Re(s) > d-1$, with a single simple pole at $s=d$ of residue $A_{\kk}(\ttheta)$, which gives the meromorphic continuation. For $\Re(s) \in [d-1+\delta, d+1]$, the two integrals are bounded by $C_2/\delta$ and by a constant, respectively. Choosing $\phi = \eta$ if $\Im(s) \geq 0$ and $\phi = -\eta$ otherwise, we have $|e^{i\phi s}| = |e^{i\phi(s-d)}| = e^{-\eta|\Im(s)|}$, which gives the bound of the lemma for $\ttheta \in K_\varepsilon$, and then for all $\ttheta$ by homogeneity.
\end{proof}

\begin{proof}[Proof of Lemma~\ref{lemma_derivativ}]
By Lemma~\ref{lemma_integral_form_gg}, $\log Zon_d(e^{-\ttheta}) = \frac{1}{2i\pi}\int_{c-i\infty}^{c+i\infty} \frac{\zeta(s+1)}{\zeta(s)}\,\Gamma(s)F_d(s,\ttheta)\, ds$ for $c>d$. On every vertical line $\Re(s) = c' > 1$, the function $\zeta(s+1)/\zeta(s)$ is bounded. Hence, by Lemma~\ref{lemma_Barnes}, we can differentiate under the integral sign, and shift the line of integration to $\Re(s) = c'$ with $\max(1,d-1) < c' < d$ (note that $d \geq 2$), the contribution of the horizontal segments vanishing thanks to the factor $e^{-\eta|\Im(s)|}$. The residue at $s=d$ gives
$$\frac{\partial^{\kk\cdot\un}}{\partial\ttheta^{\kk}} \log\left(  Zon_d(e^{-\ttheta}) \right) = \frac{\zeta(d+1)}{\zeta(d)}A_{\kk}(\ttheta) + \frac{1}{2 i \pi} \int _{c' - i\infty} ^{ c'+i \infty} \frac{\zeta(s+1)}{\zeta(s)}\, \frac{\partial^{\kk\cdot\un}}{\partial\ttheta^{\kk}}\left(\Gamma(s) F_d(s,\ttheta)\right) ds,$$
and, by Lemma~\ref{lemma_Barnes}, the last integral is $O(|\ttheta|^{-\kk\cdot\un-c'})$, which is negligible compared with $A_{\kk}(\ttheta)$, of order $|\ttheta|^{-\kk\cdot\un - d}$ under the ratio condition.
\end{proof}

The saddle-point analysis of Section~\ref{section_proof_main} also requires bounds on the derivatives of $\log Zon_d$ at complex points $e^{-\ttheta + i\aalpha}$ close to the positive real axis. They follow from Lemma~\ref{lemma_derivativ} and Cauchy's estimates. For $\boldsymbol{\tau} \in \C^d$ with $\Re(\tau_j) > 0$ for all $j$, we set
$$ \mathcal{L}_d(\boldsymbol{\tau}) = \sum_{ \vv\in \P_{d+}} 2^{d(\vv)-1} \sum_{m \geq 1} \frac{1}{m}\, e^{-m\, \boldsymbol{\tau} \cdot \vv }, $$
which is the branch of $\log Zon_d(e^{-\boldsymbol{\tau}})$ that is real on $(0,+\infty)^d$. This series converges normally on every set $\{\Re(\tau_j) \geq \eta, \; 1 \leq j \leq d\}$ with $\eta>0$, hence $\mathcal{L}_d$ is holomorphic on $\{\Re(\tau_j) >0, \; 1 \leq j \leq d\}$.

\begin{lem}\label{lemma_complex_points}
Let $\varepsilon \in (0,1)$. For all $\kk \in \N^d$, there exist constants $C>0$ and $\eta>0$ such that, for all $\ttheta \in (0,\eta)^d$ with $\varepsilon \leq \theta_i/\theta_j \leq 1/\varepsilon$ for all $i,j$, and for all $\boldsymbol{\tau} \in \C^d$ with $|\tau_j - \theta_j| \leq \frac{\theta_j}{4}$ for $1 \leq j \leq d$,
$$ \left| \frac{\partial^{\kk \cdot \un}}{ \partial \boldsymbol{\tau}^{\kk} } \mathcal{L}_d(\boldsymbol{\tau}) \right| \leq \frac{C}{|\ttheta|^{d + \kk\cdot\un}}.$$
\end{lem}

\begin{proof}
Since $|\log(1-z)| \leq -\log(1-|z|)$ for $|z|<1$, and since $|e^{-\boldsymbol{\tau}\cdot\vv}| = e^{-\Re(\boldsymbol{\tau})\cdot\vv}$ for $\vv \in \Z_+^d$, we have $|\mathcal{L}_d(\boldsymbol{\tau})| \leq \mathcal{L}_d(\Re(\boldsymbol{\tau}))$. Moreover, $\mathcal{L}_d$ is non-increasing in each variable on $(0,+\infty)^d$.
Let $\boldsymbol{\tau}$ be such that $|\tau_j - \theta_j| \leq \frac{\theta_j}{4}$ for all $j$. The closed polydisk $D$ centred at $\boldsymbol{\tau}$ with radii $\frac{\theta_j}{4}$ is contained in $\{\Re(\sigma_j) \geq \frac{\theta_j}{2}, \; 1 \leq j \leq d\}$, so that, for all $\boldsymbol{\sigma} \in D$,
$$ |\mathcal{L}_d(\boldsymbol{\sigma})| \leq \mathcal{L}_d\left(\Re(\boldsymbol{\sigma})\right) \leq \mathcal{L}_d\left(\tfrac{1}{2}\ttheta\right) \sim \frac{2^d\,\kappa_d}{\prod_{j=1}^d \theta_j}$$
by Lemma~\ref{lemma_derivativ} with $\kk = \boldsymbol{0}$ (note that $\frac{1}{2}\ttheta$ satisfies the same ratio condition as $\ttheta$). Hence $\sup_D |\mathcal{L}_d| \leq C' |\ttheta|^{-d}$ for $\ttheta$ small enough, where $C'$ only depends on $d$ and $\varepsilon$. Cauchy's estimates on $D$ then give
$$ \left| \frac{\partial^{\kk \cdot \un}}{ \partial \boldsymbol{\tau}^{\kk} } \mathcal{L}_d(\boldsymbol{\tau}) \right| \leq \kk!\, \sup_D |\mathcal{L}_d| \prod_{j=1}^d \left(\frac{4}{\theta_j}\right)^{k_j} \leq \frac{C}{|\ttheta|^{d + \kk\cdot\un}},$$
where $\kk! = k_1!\cdots k_d!$, using $\theta_j \geq \varepsilon\, |\ttheta|/\sqrt{d}$.
\end{proof}

In particular, the derivatives of $\log Zon_d$ with respect to the operators $x_j\frac{\partial}{\partial x_j} = -\frac{\partial}{\partial \tau_j}$ at the point $\xx = e^{-\ttheta + i\aalpha}$, that is $\boldsymbol{\tau} = \ttheta - i\aalpha$, are $O(|\ttheta|^{-d-\kk\cdot\un})$ uniformly for $|\alpha_j| \leq \frac{\theta_j}{4}$.

\subsection{Asymptotic expansion of the univariate generating function}\label{subsection_monovariable_equival}

We now compute an asymptotic expansion of the generating function when all the variables $\theta_i$ are equal and tend to 0, which corresponds to the dominant singularity of the generating function. To this end, we introduce the following notation.

\begin{defn}\label{definition_operator}
Let $\mathrm{M}(\C)$ be the field of meromorphic functions on $\C$. We define the polynomial $P_d(X)$ by
$$P_d(X) = \sum_{\delta = 1}^d  \binom{d}{\delta} 2^{\delta-1} \frac{\prod_{k =1}^{\delta-1}(X -k)}{(\delta-1) !} = p_{d,d-1} {X}^{d-1} + \dots + p_{d,1} {X} + p_{d,0}, $$ 
with the convention $\prod_{k =1}^{0}(X -k) = 1$, and the operator $\mathbf{\Pi}_d$ by
\begin{align*}
   \mathbf{\Pi}_d \colon \mathrm{M}(\C) &\to \mathrm{M}(\C)\\
  \phi &\mapsto p_{d,d-1}\, \phi(\cdot \: - (d-1)) + \dots + p_{d,1}\, \phi(\cdot \: - 1) + p_{d,0}\, \phi.
\end{align*}
\end{defn}

\begin{rem}
The family of polynomials $(P_d(X) )_{d\geq 1}$ satisfies the recurrence $P_{d+2} (X) = \frac{2X}{d+1}P_{d+1}(X) + P_d(X)$, with $P_1({X}) = 1$ and $P_2({X}) = 2 {X}$. As a consequence, $P_d$ is an even polynomial with constant term 1 when $d$ is odd, and an odd polynomial when $d$ is even. Moreover, $P_d$ has degree $d-1$ and leading coefficient $\frac{2^{d-1}}{(d-1)!}$. The first polynomials $P_d$ are given in Table~\ref{tab:polynomials}.
\end{rem}

\begin{table}[h]
    \centering
    \begin{tabular}{|c|c|}
        \hline
        $P_1$ & $1$ \\ 
        \hline
        $P_2$ & $2X$  \\ 
        \hline
        $P_3$ &  $2 X^2 + 1 $\\ 
        \hline
        $P_4$ & $ \frac{4}{3} X^3 + \frac{8}{3} X$\\
        \hline
    \end{tabular}
    \caption{First polynomials $P_d$.}
    \label{tab:polynomials}
\end{table}

The operator naturally arises from the Mellin transform: for $\Re(s) > d$, $\mathbf{\Pi}_d (\zeta)(s) = \sum_{k \geq 1} \frac{P_d(k)}{k^s}$, and this identity extends by meromorphic continuation.
With this notation, we have the following proposition.

\begin{figure}[ht]
    \centering
    \begin{tikzpicture}[xscale=3.2, yscale=0.13, >=stealth]
        % critical strip
        \fill[black!7] (0,-31) rectangle (1,31);
        % axes
        \draw[->, black!60] (-0.8,0) -- (2.1,0) node[right, black] {$\Re(s)$};
        \draw[->, black!60] (0,-32) -- (0,33) node[above, black] {$\Im(s)$};
        % lines Re s = c_2 and Re s = -1/2
        \draw[dashed] (1.5,-31) -- (1.5,-5);
        \draw[dashed] (1.5,1) -- (1.5,31);
        \draw[dashed] (-0.5,-31) -- (-0.5,-5);
        \draw[dashed] (-0.5,1) -- (-0.5,31);
        \node[below] at (-0.5,-0.3) {\small $-\frac{1}{2}$};
        \node[below] at (1.5,-0.3) {\small $c_2$};
        \node[below right] at (1,-0.3) {\small $1$};
        \node[below left] at (0,-0.3) {\small $0$};
        % contour gamma (here A = 0.2)
        \draw[thick, ->, variable=\t, domain=-31:8, samples=80] plot ({1-0.2/ln(2+abs(\t))},{\t});
        \draw[thick, variable=\t, domain=8:31, samples=60] plot ({1-0.2/ln(2+abs(\t))},{\t});
        \draw[thick, ->, variable=\t, domain=31:-8, samples=80] plot ({0.2/ln(2+abs(\t))},{\t});
        \draw[thick, variable=\t, domain=-8:-31, samples=60] plot ({0.2/ln(2+abs(\t))},{\t});
        \node at (0.82,27) {\small $\gamma_r$};
        \node at (0.19,27) {\small $\gamma_l$};
        % first non-trivial zeros
        \foreach \y in {14.1347, 21.0220, 25.0109, 30.4249} {
            \node at (0.5,\y) {\small $\times$};
            \node at (0.5,-\y) {\small $\times$};
        }
        % double pole at 0
        \fill (0,0) circle[x radius=0.025, y radius=0.6];
    \end{tikzpicture}
    \caption{The contour $\gamma = \gamma_r \cup \gamma_l$ (thick lines), surrounding the non-trivial zeros of $\zeta$ (crosses) in the critical strip (shaded). The dashed lines are the lines $\Re(s) = c_2$ and $\Re(s) = -\frac{1}{2}$ used in the proof of Proposition~\ref{prop_equivalent_monovariate}, and the dot is the double pole at $s=0$.}
    \label{fig:bande_critique}
\end{figure}
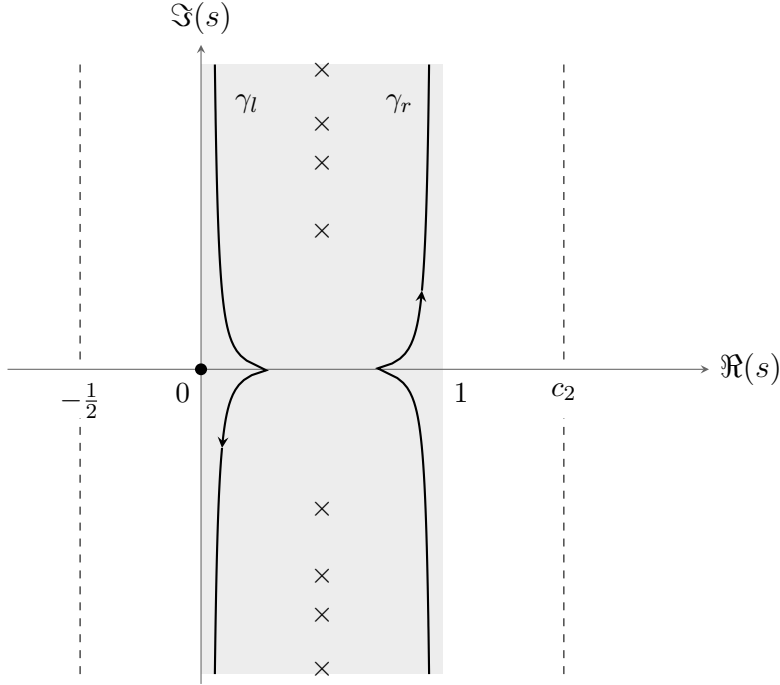

\begin{prop} \label{prop_equivalent_monovariate}
Let $A>0$ be a constant, small enough so that all the non-trivial zeros of the Riemann $\zeta$ function lie in the region $\left\{ \sigma + it, \; \frac{A}{\log(2 + |t|)} < \sigma < 1 - \frac{A}{\log(2 + |t|)}\right\}$; such a constant exists by \cite[Theorem 3.8]{Titchmarsh:riemann}. Let $\gamma$ be the contour made of the two oriented paths $\gamma_{r}(t) = 1 - \frac{A}{\log(2 + |t|)} + it $, for $t$ going from $- \infty$ to $+\infty$, and $\gamma_{l}(t) = \frac{A}{\log(2 + |t|)} + it $, for $t$ going from $+ \infty$ to $- \infty$ (see Figure~\ref{fig:bande_critique}).\medskip

We define the function $ I_{\text{crit}, d}$ by 
\begin{align}
    I_{\text{crit}, d}(\theta) = \frac{1}{2 i \pi } \int_{\gamma}  \frac{\mathbf{\Pi}_d (\zeta)(s) }{ \zeta(s)}\, \zeta(s+1) \Gamma(s) \theta^{-s} ds .
\end{align}
Then, as $\theta\rightarrow 0^+$,
\begin{align}
    \log \left(Zon_d \left(e^{-\theta}\un \right) \right) = \sum_{\delta=1}^{d-1}\left(p_{d,\delta}\frac{\zeta(\delta+2) \Gamma(\delta+1)}{\zeta(\delta+1) \theta^{\delta+1}} \right) + I_{\text{crit}, d}(\theta) +  C_d(\theta) + O \left( \theta^{1/2}\right),
\end{align}
with $C_d(\theta) = 2 \mathbf{\Pi}_d \left(\log(2\pi) \zeta -  \zeta'\right) (0) + 2 \mathbf{\Pi}_d \left(\zeta\right)(0) \log(\theta) $.
\end{prop}

\begin{proof}
We start from the logarithm of $Zon_d(e^{-\theta} \un)$ given by (\ref{generating_function}):
$$ \log\left(  Zon_d(e^{-\theta} \un) \right) = - \sum_{ \vv\in \P_{d+}} 2^{d(\vv)-1} \log \left( 1 - e^{-  \theta\, \un \cdot \vv } \right).  $$
A primitive vector $\vv \in \P_{d+}$ with $d(\vv)=\delta$ is determined by the set of its $\delta$ non-zero coordinates, which can be chosen in $\binom{d}{\delta}$ ways, and by a primitive vector of $\P_{\delta+} \cap \text{int}(\R^\delta_+)$. Expanding the logarithm in series and applying the Fubini--Tonelli theorem, we get
\begin{align*}
\log\left(  Zon_d(e^{-\theta} \un) \right) &= \sum_{\delta = 1}^{d} 2^{\delta-1} \binom{d}{\delta} \sum_{k \geq 1} \sum_{ \vv\in \P_{\delta+}\cap\, \text{int}(\R^\delta_+)} \frac{1}{k} e^{- k \theta\, \un \cdot \vv } \\
&=  \sum_{\delta = 1}^{d} 2^{\delta-1} \binom{d}{\delta} \sum_{k \geq 1} \frac{1}{k} S^{(\delta)}_{\textbf{0}}(e^{-k\theta}).\end{align*}
By the harmonic sum property of the Mellin transform (see \cite[Appendix B.7, p.~762]{Flajolet:AC}), which follows from the Fubini--Tonelli theorem, and by Proposition~\ref{proposition_eulerian} applied with $\aalpha = \textbf{0}$ in dimension $\delta$, we have, for $\Re(s) > d$,
\begin{align*}
\mathcal{M}\left[\log\left(Zon_d\left(e^{-\theta} \un\right)\right)\right](s) &= \sum_{\delta = 1}^d   \binom{d}{\delta} 2^{\delta-1} \zeta(s +1)\, \mathcal{M}\left[S^{(\delta)}_{\textbf{0}}(e^{-\theta})\right](s) \\
&= \sum_{\delta = 1}^d   \binom{d}{\delta} 2^{\delta-1} \frac{\zeta(s +1)}{\zeta(s)}\, \mathcal{M}\left[\frac {  e^{- \delta  \theta}}{(1-e^{-  \theta}) ^{\delta}}  \right](s).\end{align*}
Since $\frac {  e^{- \delta  \theta}}{(1-e^{- \theta}) ^{\delta}} = \sum_{n \geq 1} \binom{n - 1}{\delta - 1} e^{- n \theta }$, using again the harmonic sum property and the notation of Definition~\ref{definition_operator}, the transform simplifies to
\begin{align}
    \mathcal{M}\left[\log \left(Zon_d \left(e^{-\theta} \un \right)\right)\right] (s) &=\frac{1 }{ \zeta(s)} \sum_{\delta = 1}^d  \binom{d}{\delta} 2^{\delta-1} \sum_{n \geq 1}  \frac{\binom{ n - 1}{\delta-1} }{n^s}\, \zeta(s+1) \Gamma(s)  \nonumber  \\
    &=\frac{\mathbf{\Pi}_d (\zeta)(s) }{ \zeta(s)}\, \zeta(s+1) \Gamma(s).
\end{align}
To give a concrete idea of this formula, here are the Mellin transforms in dimensions 2, 3, and 4:
\begin{align*}
    \mathcal{M}\left[\log \left(Zon_2 \left(e^{-\theta} \un\right)\right)\right] (s) &=\frac{2 \zeta(s - 1) }{ \zeta(s)}\, \zeta(s+1) \Gamma(s),\\
    \mathcal{M}\left[\log \left(Zon_3 \left(e^{-\theta} \un \right)\right)\right] (s) &=\frac{2 \zeta(s - 2) + \zeta(s) }{ \zeta(s)}\, \zeta(s+1) \Gamma(s),\\
    \mathcal{M}\left[\log \left(Zon_4 \left(e^{-\theta} \un \right)\right)\right] (s) &=\frac{ \frac{4}{3}\zeta(s -3) + \frac{8}{3}\zeta(s -1) }{ \zeta(s)}\, \zeta(s+1) \Gamma(s).
\end{align*}

This function extends to a meromorphic function on $\C$, and we use the residue theorem iteratively to collect, from right to left, the poles of the Mellin transform in the inverse Mellin formula. On the real line, these poles are located at the integers $\delta+1$ such that $1 \leq \delta \leq d-1$ and $p_{d,\delta} \neq 0$, and at $0$. Note that $s=1$ is not a pole, since the pole of $p_{d,0}\,\zeta(s)$ is cancelled by the zero of $1/\zeta(s)$. Because of the denominator, the non-trivial zeros of the $\zeta$ function are also poles of the transform. \medskip 
 
To begin with, the Mellin inversion formula holds for $c> d$ and gives
$$ \log(Zon_d (e^{-\theta} \un )) = \frac{1}{2i\pi} \int_{c - i \infty}^{c + i \infty}\frac{\mathbf{\Pi}_d (\zeta)(s) }{ \zeta(s)}\, \zeta(s+1)\Gamma(s)  \theta^{-s}  d s.$$

Since $s= d$ is a simple pole of the Mellin transform, the residue theorem allows us to shift the line of integration to the left: for all $c_1\in (d-1, d)$,
$$ \log(Zon_d (e^{-\theta} \un )) = \frac{p_{d, d-1} \zeta(d +1) \Gamma(d) }{\zeta(d) \theta^{d}} + \frac{1}{2i\pi} \int_{c_1- i\infty}^{c_1+ i\infty} \frac{\mathbf{\Pi}_d (\zeta)(s) }{ \zeta(s)}\, \zeta(s+1) \Gamma(s) \theta^{-s} d s.$$
We repeat the process until $1 < c_2 <2$. Then recall the contour $\gamma$ defined in the statement of the proposition. On $\gamma$, $|\theta^{-s}| \leq \theta^{-1}$ for $\theta < 1$, $\Gamma$ decays exponentially, and $1/\zeta(s) = O(\log|\Im(s)|)$ on $\gamma_r$ for $A$ small enough \cite[Section 3.11]{Titchmarsh:riemann}; on $\gamma_l$, the functional equation $\zeta(s) = \chi(s)\zeta(1-s)$, with $|\chi(\sigma+it)| \asymp |t|^{1/2-\sigma}$, shows that $1/\zeta(s)$ also grows at most polynomially in $|\Im(s)|$. Hence the integral defining $I_{\text{crit},d}(\theta)$ converges absolutely. The same holds on the line $\Re(s) = -\frac{1}{2}$, where $|\zeta(s)| \asymp |\Im(s)|\,|\zeta(\frac{3}{2} - s)|$ is bounded from below by a positive multiple of $|\Im(s)|$ for $|\Im(s)| \geq 1$.\medskip

We use the following result of Valiron \cite[Theorem 9.7]{Titchmarsh:riemann}:
there exist $\alpha >0$ and a sequence $(T_k)$ such that, for all $ k \in \N$, $ k < T_k < k + 1$  and $|\zeta(s)| > T_k^{-\alpha}$ for all $s$ such that  $ -1 \leq \Re(s) \leq 2$  and $|\Im(s)| = T_k$. Applying the residue theorem to the positively oriented rectangle with vertices
$\frac{3}{2} \pm i T_k$ and $-\frac{1}{2} \pm i T_k$, and letting $k$ tend to $+\infty$, the contribution of the horizontal segments tends to $0$. As a result, the integral on the line $\Re(s) = c_2$ is equal to the sum of the integral on the line $\Re(s) = - \frac{1}{2}$, of the residue at $0$, and of $I_{\text{crit},d}(\theta)$. The residue at the double pole $s=0$ is $C_d(\theta)$, using $\zeta(0) = -\frac{1}{2}$ and $\zeta'(0) = -\frac{1}{2}\log(2\pi)$. We get
\begin{align*}
    \log(Zon_d (e^{-\theta} \un )) =  &\sum_{\delta=1}^{d-1}\left(p_{d,\delta}\frac{\zeta(\delta+2) \Gamma(\delta+1)}{\zeta(\delta+1) \theta^{\delta+1}} \right) + I_{\text{crit}, d} (\theta)+ C_d(\theta) \\
    &+\frac{1}{2i\pi} \int_{ -\frac{1}{2}- i \infty}^{- \frac{1}{2} + i \infty} \frac{\mathbf{\Pi}_d (\zeta)(s) }{\zeta(s)}\,\zeta(s+1)\Gamma(s)  \theta^{-s} {d}s. 
\end{align*}
Since $|\theta^{-s}| = \theta^{\frac{1}{2}}$ on the line $\Re(s) = - \frac{1}{2}$, the last integral is $O \left(\theta^{\frac{1}{2}}\right)$.
\end{proof}

%%%%%%%%%%%%%%%%%%%%%%%%%%%%%%%%%%%%%%%%%%%%%%%%%%%

\section{Proof of the main theorem}\label{section_proof_main}

The proof of Theorem~\ref{theorem_1} relies on the saddle-point method in $d$ dimensions (see \cite[Theorem VIII.3]{Flajolet:AC} for the one-dimensional case). The saddle-point method extracts the coefficients of a series from Cauchy's integral formula. Recall that $Zon_d$ is analytic on the open unit polydisk. Hence, for all $\nn \in \N^d$ and all $\ttheta \in (0,+\infty)^d$, denoting by $[\xx^{\nn}]F(\xx)$ the coefficient of $\xx^{\nn}$ in the series $F$, we have
$$[\xx^{\nn}]\,Zon_d (\xx) = \frac{1}{(2 i \pi)^d} \oint_{|x_1| = e^{-\theta_1}} \cdots \oint_{|x_d| = e^{-\theta_d}} \frac{Zon_d (\xx)}{x_1^{n_1+1}\cdots x_d^{n_d+1}}\, dx_1\cdots dx_d.$$
The point of the saddle-point method is to choose $\ttheta$ such that the integral is asymptotically a Gaussian integral that can be computed. In this context, the Gaussian approximation (Proposition~\ref{definition_gittenberger}) is closely related to the probabilistic local limit theorem (see \cite{Bureaux:LLT}). \medskip

\subsection{Admissibility of a multivariate saddle-point method}

Since the generating function is studied on the Cartesian product of $d$ open unit complex disks, we write, for $1 \leq j \leq d$, $x_j = e^{-\theta_j + i \alpha_j}$, with $\theta_j >0$ and $\alpha_j \in [- \pi, \pi]$. In vector notation, $\xx = e^{- \ttheta + i \aalpha}$. With this parametrization, Cauchy's formula becomes
$$[\xx^{\nn}]\,Zon_d (\xx) = \frac{e^{\nn\cdot \ttheta}}{(2 \pi)^d} \int_{[-\pi,\pi]^d} Zon_d\left(e^{- \ttheta + i \aalpha}\right) e^{-i\,\nn \cdot \aalpha}\, d\aalpha. $$

In what follows, the notation $\xx \rightarrow \textbf{l}$ means that each component $x_i$ tends to $l_i$. We denote by $\ba(\xx) = \left(a_i(\xx ) \right)_{1 \leq i\leq d}$, $\BB (\xx) = \left(B_{i,j}(\xx )\right)_{1 \leq i,j \leq d}$, and $\CC(\xx) = \left(C_{i,j,k}(\xx )\right)_{1 \leq i, j, k \leq d}$ the families of first, second, and third order partial derivatives of $\log Zon_d$ with respect to the operators $x_j \frac{\partial}{\partial x_j}$. For instance, for $1 \leq j \leq d$,
$$a_j(\xx) = \frac{x_j \frac{\partial  }{\partial x_j }Zon_d(\xx)}{Zon_d(\xx) }. $$
With $\xx = e^{-\ttheta}$, we have $a_j(e^{-\ttheta}) = - \frac{\partial}{\partial \theta_j} \log Zon_d(e^{-\ttheta})$, $B_{i,j}(e^{-\ttheta}) = \frac{\partial^2}{\partial \theta_i \partial\theta_j} \log Zon_d(e^{-\ttheta})$, and so on.
Recall that $p_{\nn}$ is the number of zonotopes inscribed in the box $[0,n_1] \times \dots \times [0,n_d]$. The following result is often called the H-admissibility property.

\begin{prop}\label{definition_gittenberger}
 As $\ttheta \rightarrow \boldsymbol{0}$ with $\theta_1 = \dots = \theta_d$, we have, uniformly for $\nn \in \N^d$,
 $$p_{\nn} = \frac{Zon_d(e^{- \ttheta})}{e^{- \nn \cdot \ttheta} \sqrt{(2 \pi)^d \det \BB( e^{- \ttheta})}}\left( \exp\left( - \frac{1}{2}(\ba(e^{- \ttheta}) - \nn)^{\intercal}\BB(e^{- \ttheta})^{-1}(\ba(e^{- \ttheta}) - \nn) \right) + o(1) \right). $$
\end{prop}
This proposition relies on a few technical lemmas, which check the five conditions (labelled in the lemmas) of \cite[Definition 2]{Gittenberger:hadmi}. Gittenberger and Mandlburger consider entire functions, and let the radii $\boldsymbol{r}$ tend to $(+\infty,\dots,+\infty)$; however, the proof of \cite[Theorem 4]{Gittenberger:hadmi} only uses Cauchy's integral formula on the torus of radii $\boldsymbol{r}$, conditions (I) and (II) at the point $\boldsymbol{r}$, and \cite[Lemma 1]{Gittenberger:hadmi}, which follows from (I), (II), and (III). It therefore applies verbatim to a function analytic in the open unit polydisk, along any path $\boldsymbol{r} \rightarrow \un$ on which these conditions hold; here, the path is the diagonal $\boldsymbol{r} = e^{-\theta}\un$, $\theta \rightarrow 0^+$, and all the statements below are only needed (and stated) for $\theta_1 = \dots = \theta_d$, although the proofs work as soon as the ratios $\theta_i/\theta_j$ remain bounded. As in \cite[Definition 2]{Gittenberger:hadmi}, the domain $\Delta(\ttheta)$ on which the Gaussian approximation holds is a cuboid spanned by an orthonormal basis of eigenvectors of $\BB(e^{-\ttheta})$.

\begin{lem}\label{condition1} (\textbf{I}) $\BB(e^{-\ttheta})$ is positive definite.\medskip

Let $\beta $ be a real number in the interval $ \left(1 + \frac{d}{3} ,1 + \frac{d}{2}\right)$, let $\boldsymbol{w}_1(\ttheta),\dots,\boldsymbol{w}_d(\ttheta)$ be an orthonormal basis of eigenvectors of $\BB(e^{-\ttheta})$, and define the cuboid
$$\Delta (\ttheta ) = \left\{ \sum_{j=1}^d \mu_j \boldsymbol{w}_j(\ttheta) \; \middle| \; |\mu_j| \leq \left(\underset{1 \leq k \leq d}{\max}\, \theta_k \right)^\beta \text{ for } 1 \leq j\leq d \right\}.$$
Then, uniformly for $\aalpha \in \Delta (\ttheta )$, as $\ttheta \rightarrow \boldsymbol{0}$,
$$Zon_d(e^{- \ttheta + i \aalpha}) =
Zon_d(e^{- \ttheta}) 
\exp \left( i\aalpha^{\intercal} \ba(e^{- \ttheta}) -
\frac{\aalpha^{\intercal} \BB(e^{- \ttheta}) \aalpha }{2 }
\right)(1 + o(1)).$$
\end{lem}

\begin{proof}
By Lemma~\ref{lemma_derivativ}, we have the following asymptotic equivalents of $\ba(e^{- \ttheta})$, $\BB (e^{- \ttheta})$, and $\CC(e^{- \ttheta})$, where $\delta_{i,j}$ denotes Kronecker's symbol and $\delta_{i,j,k} =1 $ if $i = j = k $ and 0 otherwise:
    \begin{align}\label{equivalent_aa}
        a_i(e^{- \ttheta}) \underset{\ttheta \rightarrow \boldsymbol{0}}{\sim} \kappa_d \: \frac{1}{\theta_i  \prod_{j = 1}^d \theta_j},
    \end{align}
    \begin{align}\label{equivalent_BB}
       B_{i,j} (e^{- \ttheta}) \underset{\ttheta \rightarrow \boldsymbol{0}}{\sim} (1 + \delta_{i,j}) \: \kappa_d \: \frac{1}{\theta_i \theta_j  \prod_{k = 1}^d \theta_k},
    \end{align}
    $$  C_{i,j,k} (e^{- \ttheta}) \underset{\ttheta \rightarrow \boldsymbol{0}}{\sim} (1 + \delta_{i,j} + \delta_{j,k} + \delta_{k, i} + 2 \delta_{i,j,k})  \: \kappa_d \:\frac{1}{\theta_i \theta_j \theta_k  \prod_{l = 1}^d \theta_l}. $$
In particular, $\BB(e^{-\ttheta})$ is asymptotically equivalent to $\frac{\kappa_d}{\prod_k \theta_k} D^{-1}(I_d + J_d)D^{-1}$, where $D = \operatorname{diag}(\theta_1,\dots,\theta_d)$, $I_d$ is the identity matrix and $J_d$ the all-ones matrix; since $I_d + J_d$ is positive definite, so is $\BB(e^{-\ttheta})$ for $\ttheta$ small enough.
To conclude with the Lagrange form of Taylor's theorem, it suffices to check that, for any $t \in (0,1)$,
$$ C_{i,j,k}(e^{-\ttheta + i t \aalpha})\, \alpha_i \alpha_j \alpha_k \underset{\ttheta \rightarrow \boldsymbol{0}}{\longrightarrow} 0, \text{ uniformly for }\aalpha \in \Delta(\ttheta),$$
which holds since, for $\aalpha \in \Delta(\ttheta)$, $\|\aalpha\| \leq \sqrt{d}\,(\max_k \theta_k)^\beta$; this is at most $\frac{1}{4}\min_j\theta_j$ for $\ttheta$ small, as $\beta > 1$, so that Lemma~\ref{lemma_complex_points} gives $C_{i,j,k}(e^{-\ttheta + it\aalpha}) = O(|\ttheta|^{-(d+3)})$, and $3\beta > d+3$.
\end{proof}

\begin{lem}\label{condition2} (\textbf{II})
As $\ttheta \rightarrow \boldsymbol{0}$, uniformly for $\aalpha \in [-\pi,\pi]^d \setminus  \Delta (\ttheta)$,
$$|Zon_d(e^{ - \ttheta + i \aalpha})| = o \left( \frac{Zon_d(e^{-\ttheta})}{\sqrt{ \det \BB(e^{-\ttheta})}} \right).$$
\end{lem}

\begin{proof}
    Let $\aalpha \notin \Delta(\ttheta)$. We start from
    $$\log \left(  \frac{|Zon_d(e^{ - \ttheta + i \aalpha} )|}{Zon_d(e^{ - \ttheta })} \right) = 
     - \sum_{ \vv\in \P_{d+}} 2^{d(\vv)-1} \log \left(  \frac{|1 - e^{-  \ttheta \cdot \vv + i \aalpha \cdot \vv}| }{1 - e^{-  \ttheta \cdot \vv } }\right). $$
    Using $ \frac{\left|1 - x e^{iy}\right|} {1 - x} = \sqrt{1+ \frac{4 x \sin^2\left(y/2\right) }{(1 - x)^2}}$ for $x \in (0,1)$, we have
    $$\log \left(  \frac{|Zon_d(e^{ - \ttheta + i \aalpha} )|}{Zon_d(e^{ - \ttheta })} \right) = 
     - \sum_{ \vv\in \P_{d+}} 2^{d(\vv)-2} \log \left(1 +  \frac{ 4 e^{-\ttheta\cdot \vv} \sin^2\left(\frac{\aalpha \cdot \vv}{2} \right) }{(1 - e^{-  \ttheta \cdot \vv })^2 }\right).$$
    Since the denominator is smaller than 1, and since $\frac{\log(5)}{4}x \leq \log(1 + x)$ for $0 \leq x \leq 4$, we get
    $$\log \left(  \frac{|Zon_d(e^{ - \ttheta + i \aalpha} )|}{Zon_d(e^{ - \ttheta })} \right) \leq
     -  \log(5) \sum_{ \vv\in \P_{d+}} 2^{d(\vv)-2}  e^{-\ttheta \cdot \vv} \sin^2\left(\frac{\aalpha \cdot \vv}{2}\right) = - \log(5)\, U_{\ttheta, \aalpha}.$$
   Set $\theta = \max_{1 \leq j \leq d}\theta_j$ and $r = \max_{1 \leq j \leq d}|\alpha_j|$. Since $\aalpha \notin \Delta(\ttheta)$ and the basis $(\boldsymbol{w}_j)$ is orthonormal, $\|\aalpha\| > \theta^\beta$, hence $r > \theta^\beta/\sqrt{d}$. Let $c_0>0$ be a small constant, fixed below. We distinguish two regimes.

   \emph{First regime: $\theta^\beta/\sqrt{d} < r \leq c_0 \theta$.} By Taylor's formula with integral remainder,
   $$\log Zon_d(e^{-\ttheta+i\aalpha}) - \log Zon_d(e^{-\ttheta}) = i\aalpha^{\intercal}\ba(e^{-\ttheta}) - \frac{1}{2}\aalpha^{\intercal}\BB(e^{-\ttheta})\aalpha + R,$$
   with $|R| \leq \frac{1}{6}\sup_{t \in [0,1]} \sum_{i,j,k} |C_{i,j,k}(e^{-\ttheta + it\aalpha})|\, r^3$.
   By (\ref{equivalent_BB}), and by Lemma~\ref{lemma_complex_points} for $\CC$ at the complex points $e^{-\ttheta + it\aalpha}$ (which applies as soon as $c_0 \leq \frac{1}{4}$, since $\ttheta = \theta\un$ here), there are constants $c_1, c_2 >0$ such that $\aalpha^{\intercal}\BB(e^{-\ttheta})\aalpha \geq c_1 \theta^{-(d+2)} r^2$ and $|R| \leq c_2\theta^{-(d+3)} r^3$. Choosing $c_0 \leq \min\left(\frac{1}{4}, \frac{c_1}{4c_2}\right)$, we get
   $$\log \left(  \frac{|Zon_d(e^{ - \ttheta + i \aalpha} )|}{Zon_d(e^{ - \ttheta })} \right) \leq -\frac{c_1}{4}\,\theta^{-(d+2)} r^2 \leq -\frac{c_1}{4d}\,\theta^{-(d+2-2\beta)},$$
   which tends to $-\infty$ as a negative power of $\theta$, since $\beta < 1 + \frac{d}{2}$.

   \emph{Second regime: $r \geq c_0 \theta$.} Without loss of generality, $|\alpha_1| = r$. Let $M = \lfloor 1/\theta \rfloor$ and consider the primitive vectors $\vv_i = (i, 1, 0, \dots, 0)$, for $1 \leq i \leq M$, for which $d(\vv_i) = 2$, $e^{-\ttheta\cdot\vv_i} \geq e^{-(M+1)\theta} \geq e^{-2}$, and
   $$ \sin^2\left(\frac{\aalpha \cdot \vv_i}{2}\right) = \sin^2(\varphi_i), \qquad \varphi_i = \frac{i \alpha_1 +\alpha_2 }{2}. $$
   The points $\varphi_i$ form an arithmetic progression with step $h = |\alpha_1|/2 \in [c_0\theta/2, \pi/2]$. Let $\varrho = \min(c_0/40, \pi/8)$ and call $i$ \emph{bad} if $\varphi_i$ lies at distance less than $\varrho$ from $\pi\Z$. If $h \geq 2\varrho$, two consecutive indices cannot both be bad, since $2\varrho \leq h \leq \pi - 2\varrho$; hence at least $\lfloor M/2 \rfloor$ indices are good. If $h < 2\varrho$, each interval of length $2\varrho$ centred at a point of $\pi\Z$ contains at most $2\varrho/h + 1 \leq 4\varrho/h$ points of the progression, and the progression meets at most $Mh/\pi + 2$ such intervals, so that the number of bad indices is at most
   $$\left(\frac{Mh}{\pi} + 2\right)\frac{4\varrho}{h} = \frac{4\varrho}{\pi}M + \frac{8\varrho}{h} \leq \left(\frac{4\varrho}{\pi} + \frac{16\varrho}{c_0}\right)\frac{1}{\theta} \leq \frac{1}{2\theta}.$$
   In both cases, at least $M/3$ indices are good for $\theta$ small enough, and each of them satisfies $\sin^2(\varphi_i) \geq \sin^2\varrho$. Therefore
   $$ U_{\ttheta, \aalpha} \geq  \sum_{i = 1}^{M}  e^{-\ttheta \cdot \vv_i}  \sin^2\left(\frac{\aalpha \cdot \vv_i}{2}\right) \geq \frac{M}{3}\, e^{-2} \sin^2\varrho \geq \frac{c}{\theta}$$
   for some constant $c>0$, which gives
   $$|Zon_d(e^{ - \ttheta + i \aalpha} )| \leq Zon_d(e^{ - \ttheta }) \exp\left(- c \log(5)\, \theta^{-1}\right).$$

   In both regimes, $|Zon_d(e^{ - \ttheta + i \aalpha})|/Zon_d(e^{-\ttheta})$ decays faster than any power of $\theta$, uniformly for $\aalpha \notin \Delta(\ttheta)$.
   By (\ref{equivalent_BB}), the determinant of $\BB(e^{-\ttheta})$ satisfies
   $$ \det\BB(e^{-\ttheta}) \sim (d+1)\, \kappa_d^d \: \frac{1}{ \prod_{k = 1}^d \theta_k^{d+2}}.$$
    It follows that 
    $$\frac{1}{\sqrt{\det \BB(e^{-\ttheta})}} = O \left(\left(\max_{1 \leq j \leq d }\theta_j\right)^{ \frac{d(d+2)}{2}} \right),$$
    and we get the uniform estimate of the lemma. 
\end{proof}

\begin{lem}\label{condtion345}
    The following properties hold:
    \begin{itemize}
        \item(\textbf{III}) The eigenvalues $ \lambda_1(\ttheta),\dots, \lambda_d(\ttheta)$ of $\BB(e^{-\ttheta})$ all tend to $+\infty$ as $\ttheta \rightarrow \boldsymbol{0}$.
        \item(\textbf{IV}) $B_{i,i}(e^{-\ttheta}) = o (a_i(e^{-\ttheta})^2)$ as $\ttheta \rightarrow \boldsymbol{0}$.
        \item (\textbf{V}) For $\ttheta$ small enough and $\aalpha \in [- \pi, \pi]^d \setminus \{\boldsymbol{0}\} $, we have  $|Zon_d(e^{-\ttheta + i \aalpha})| < Zon_d(e^{-\ttheta})$.
    \end{itemize}
\end{lem}

\begin{proof}
Property (\textbf{III}) follows from (\ref{equivalent_BB}), since the smallest eigenvalue of $\frac{\kappa_d}{\prod_k \theta_k} D^{-1}(I_d + J_d)D^{-1}$ is at least $\frac{\kappa_d}{\prod_k \theta_k \max_j \theta_j^2}$. Property (\textbf{IV}) directly follows from (\ref{equivalent_aa}) and (\ref{equivalent_BB}), since $B_{i,i}/a_i^2 \sim \frac{2}{\kappa_d}\prod_k \theta_k$. Property (\textbf{V}) follows from the first formula in the proof of Lemma~\ref{condition2}: all the terms of the sum are non-negative, and the term corresponding to the $j$-th vector of the canonical basis is positive as soon as $\alpha_j \neq 0$.
\end{proof}

\begin{proof}[Proof of Proposition~\ref{definition_gittenberger}]
By Lemmas~\ref{condition1}, \ref{condition2}, and \ref{condtion345}, the function $Zon_d$ satisfies the five conditions of \cite[Definition 2]{Gittenberger:hadmi}, hence \cite[Theorem 4]{Gittenberger:hadmi} applies.
\end{proof}

\subsection{The saddle-point equation}
The last step towards the main theorem is the choice of $\ttheta$. Let $\nn \in (\N^*)^d$ be the dimensions of the box. In order to compute the number $p_{\nn}$ of zonotopes inscribed in this box, we choose $\ttheta$ as an approximate solution of the so-called saddle-point equation $\ba(e^{-\ttheta}) = \nn$.

\begin{lem}\label{saddle_equation_solution}
The vector $\widetilde{\ttheta}_{\nn}$ defined by 
$$ \widetilde{\theta}_{\nn, i} = \kappa_d^{\frac{1}{d+1}}\, \frac{\left(\prod_{j=1}^d n_j \right)^{\frac{1}{d+1}}}{n_i}, \qquad 1 \leq i \leq d, $$
satisfies $\ba(e^{- \widetilde{\ttheta}_{\nn}}) =  \nn\, (1 + o(1) )$.
\end{lem}

\begin{proof}
By (\ref{equivalent_aa}), for $1 \leq i \leq d$,
$$  a_i(e^{- \ttheta}) - n_i \underset{\ttheta \rightarrow \boldsymbol{0}}{=} \frac{\kappa_d}{\theta_i  \prod_{j = 1}^d \theta_j} -n_i + o\left( \frac{1}{\theta_i  \prod_{j = 1}^d \theta_j} \right). $$
Setting $\frac{\kappa_d}{\theta_i  \prod_{j} \theta_j} = n_i$ for all $i$ and taking the product over $i$, we obtain
$$  \prod_{j = 1}^d \theta_j = \left(\frac{\kappa_d^d}{\prod_{j =1}^{d} n_j}\right)^{\frac{1}{d+1}},$$
and the result follows.
\end{proof}

Lemma~\ref{saddle_equation_solution} is not precise enough to get rid of the Gaussian factor in Proposition~\ref{definition_gittenberger}: this requires $(\ba - \nn)^{\intercal}\BB^{-1}(\ba - \nn) \rightarrow 0$, that is $|\ba(e^{-\ttheta}) - \nn| = o(|\ttheta|^{-(d+2)/2})$, whereas the error in Lemma~\ref{saddle_equation_solution} is only $o(|\ttheta|^{-(d+1)})$. We therefore restrict ourselves to the diagonal $\nn = n\un$, and use the exact solution of the saddle-point equation. Let $\ell_d(\theta) = \log Zon_d(e^{-\theta}\un)$. By symmetry, $a_i(e^{-\theta\un}) = -\frac{1}{d}\,\ell_d'(\theta)$ for $1 \leq i \leq d$, and this function decreases from $+\infty$ to $0$ on $(0,+\infty)$. Hence there is a unique $\theta^*_n > 0$ such that $\ba(e^{-\theta^*_n\un}) = n\un$, and $\theta^*_n \sim \tilde{\theta}_n = (\kappa_d/n)^{1/(d+1)}$ by Lemma~\ref{saddle_equation_solution}. Proposition~\ref{definition_gittenberger}, with $\ttheta = \theta^*_n\un$ (see also \cite[Corollary 3]{Gittenberger:hadmi}), gives
\begin{align}\label{p_n_equivalent}
    z_d(n) = p_{n\un} \underset{ n \rightarrow + \infty}{\sim} \frac{\exp\left(\ell_d(\theta^*_n) + d n \theta^*_n\right)}{\sqrt{(2 \pi)^d \det \BB \left( e^{- \theta^*_n\un} \right)}} .
\end{align}
As Sinai did in dimension 2 \cite{Sinai:polygonallines}, the same approach could lead to an estimate of the number of zonotopes inscribed in any box $[0,c_1 n] \times \dots \times [0, c_d n]$ (with positive constants $c_i$). However, a more thorough analysis (analogous to Section~\ref{subsection_monovariable_equival}) would be needed to obtain an asymptotic expansion of $Zon_d$ off the diagonal, so we restrict ourselves to $z_d(n) = p_{(n, \dots,n)}$.

It remains to expand the exponent in (\ref{p_n_equivalent}). Following Proposition~\ref{prop_equivalent_monovariate}, we write $\ell_d = L_d + R_d$, with
$$ L_d(\theta) = \sum_{\delta=1}^{d-1} p_{d,\delta}\,\frac{\zeta(\delta+2)\, \delta!}{\zeta(\delta+1)}\, \theta^{-(\delta+1)}, \qquad R_d(\theta) = I_{\text{crit},d}(\theta) + C_d(\theta) + O(\theta^{1/2}),$$
and we define
\begin{align}\label{def_F_d}
F_d(n) = \min_{\theta > 0}\left(L_d(\theta) + dn\theta\right).
\end{align}
The minimum is attained at a unique point $\hat{\theta}_n \sim \tilde{\theta}_n$, which can be expanded as a series in powers of $\tilde{\theta}_n^{2}$. Plugging this series into $L_d(\theta) + dn\theta$ shows that $F_d(n)$ is the sum of a polynomial $\widehat{Q}_d(n^{1/(d+1)})$ without constant term, of a constant $\gamma_d$, and of a term $o(1)$.

\begin{lem}\label{lemma_exponent}
As $n \rightarrow \infty$, $\ell_d(\theta^*_n) + dn\theta^*_n = F_d(n) + R_d(\tilde{\theta}_n) + o(1)$.
\end{lem}

\begin{proof}
The function $f(\theta) = L_d(\theta) + dn\theta$ satisfies $f''(\theta) \asymp \theta^{-(d+2)}$ near $\hat{\theta}_n$. Since $\theta^*_n$ minimizes $f + R_d$, we have $f(\theta^*_n) + R_d(\theta^*_n) \leq F_d(n) + R_d(\hat{\theta}_n)$ and $|\theta^*_n - \hat{\theta}_n| = O(|R_d'| / f'') $, so that
$$\left|\ell_d(\theta^*_n) + dn\theta^*_n - F_d(n) - R_d(\hat{\theta}_n)\right| = O\left(\frac{\sup |R_d'|^2}{f''(\hat{\theta}_n)}\right) = O\left(\sup|R_d'|^2\, \tilde{\theta}_n^{\,d+2}\right),$$
where the suprema are taken on a neighbourhood of $\tilde\theta_n$ of relative size $o(1)$. It remains to check that $R_d'(\theta) = o(\theta^{-2})$, which makes this $o(\tilde{\theta}_n^{\,d-2}) = o(1)$. The derivative of $C_d(\theta)$ is $O(\theta^{-1})$, and differentiating under the integral sign shows that the derivative of the integral over $\Re(s) = -\frac{1}{2}$ in the proof of Proposition~\ref{prop_equivalent_monovariate} is $O(\theta^{-1/2})$. Finally,
$$I_{\text{crit},d}'(\theta) = -\frac{1}{2i\pi}\int_\gamma \frac{\mathbf{\Pi}_d (\zeta)(s) }{ \zeta(s)}\, \zeta(s+1)\,\Gamma(s)\, s\, \theta^{-s-1}\, ds,$$
and, for $s = \gamma_r(t)$ or $s = \gamma_l(t)$ and $\theta<1$, $|\theta^{-s-1}| \leq \theta^{-2}\, \theta^{A/\log(2+|t|)}$. For each $t$, $\theta^{A/\log(2+|t|)} \rightarrow 0$ as $\theta \rightarrow 0$, and this factor is bounded by 1; since the rest of the integrand is integrable on $\gamma$ (see the proof of Proposition~\ref{prop_equivalent_monovariate}), the dominated convergence theorem gives $I_{\text{crit},d}'(\theta) = o(\theta^{-2})$. Finally, $\hat{\theta}_n - \tilde{\theta}_n = O(\tilde{\theta}_n^3)$, so that $R_d(\hat{\theta}_n) - R_d(\tilde{\theta}_n) = o(\tilde{\theta}_n) = o(1)$.
\end{proof}

\subsection{Main theorem}\label{section_resultat_final}

\begin{thm}\label{theorem_1_complet}
Let $z_d(n)$ be the number of lattice zonotopes inscribed in $[0,n]^d$. Recall that $\kappa_d = \frac{2^{d-1} \zeta(d+1)}{\zeta(d)}$, and let $\mathbf{\Pi}_d $ and $(p_{d,\delta})_{0 \leq \delta < d}$ respectively be the operator and the coefficients introduced in Definition~\ref{definition_operator}. With $\gamma$ the contour defined in Proposition~\ref{prop_equivalent_monovariate}, let
\begin{align}
    I_{\text{crit}, d}(\theta) = \frac{1}{2 i \pi } \int_{\gamma}  \frac{\mathbf{\Pi}_d (\zeta)(s) }{ \zeta(s)}\, \zeta(s+1)\, \Gamma(s)\, \theta^{-s} ds, 
\end{align}
and $I_{\text{crit}}(d,n) = I_{\text{crit}, d}\left( \left(\kappa_d/n\right)^{\frac{1}{d+1}} \right)$. Let $\widehat{Q}_d$ and $\gamma_d$ be the polynomial and the constant defined after (\ref{def_F_d}). As $n $ tends to $\infty$, we have 
$$z_d(n) \sim  \alpha_d\, e^{\gamma_d}\, n ^{\beta_d} \exp \left( \widehat{Q}_d(n^\frac{1}{d+1}) + I_{\text{crit}}(d,n) \right), $$
with $\alpha_d = \frac{\kappa_d^{\frac{d}{2(d+1)} +  \frac{2}{d+1} \mathbf{\Pi}_d [\zeta](0) } \exp \left( 2\mathbf{\Pi}_d \left[\log(2\pi)\zeta -  \zeta' \right](0) \right) }{ (2\pi)^{d/2 } \sqrt{d+1}}$ and 
$\beta_d = -\frac{1}{2(d+1)}\left(d(d+2) + 4 \mathbf{\Pi}_d \left[\zeta\right] (0) \right)$.

Moreover, $\widehat{Q}_d$ has degree $d$, and $\widehat{Q}_d(X) = Q_d(X) - E_d(X)$, where
\begin{align}
    Q_d(X) = (d+1) \kappa_d^{\frac{1}{d+1}}  X^d +  \sum_{\delta=2}^{d-1} p_{d,\delta-1}\frac{\zeta(\delta+1) (\delta-1)!}{\zeta(\delta) }\kappa_d^{-\frac{\delta}{d+1}} X^{\delta}
\end{align}
is the value of $L_d(\theta) + dn\theta$ at $\theta = \tilde{\theta}_n$, and $E_d$ is a polynomial without constant term, of degree at most $d-4$. In particular $E_d = 0$ and $\gamma_d=0$ for $d \in \{2,3\}$, $E_4 = 0$ and $\gamma_4 = -\frac{8\zeta(3)^2}{225\,\zeta(5)}$, and for $d \geq 5$ the leading term of $E_d$ is
$$\frac{(d-2)^2\, \lambda_d^2}{2d(d+1)}\, \kappa_d^{-\frac{2d-3}{d+1}}\, X^{d-4}, \qquad \text{where } \lambda_d = p_{d,d-3}\,\frac{\zeta(d-1)\,(d-3)!}{\zeta(d-2)}.$$
For instance, $E_5(X) = \frac{7\cdot 2^{1/3}\sqrt{3}\; 35^{1/6}\,\pi\,\zeta(5)^{7/6}}{288\,\zeta(3)^2}\, X \approx 0.21758\, X$. For $d \geq 6$, $E_d$ and $\gamma_d$ are explicitly computable from the expansion of $\hat\theta_n$ described in the proof below.

Finally, under the hypothesis that all the zeros of the Riemann $\zeta$ function in the critical strip are simple, $I_{\text{crit}}(d,n)$ can be written as a sum over the set $Z$ of non-trivial zeros:
\begin{align}\label{I_crit_somme}
I_{\text{crit}}(d,n)  = \sum_{r\in Z} \frac{\mathbf{\Pi}_d [\zeta](r)\, \zeta(r+1)\,\Gamma(r)}{\zeta'(r)}  \left( \frac{n}{\kappa_d} \right)^{ \frac{r}{d+1}},  
\end{align}
where the sum is understood as the limit of the partial sums over the zeros $r$ with $|\Im(r)| < T_k$, as $k \rightarrow \infty$ (with $(T_k)$ as in the proof of Proposition~\ref{prop_equivalent_monovariate}).
\end{thm}

\begin{proof}
On the diagonal, (\ref{equivalent_BB}) gives $\det \BB(e^{-\theta\un}) \sim (d+1)\kappa_d^d\, \theta^{-d(d+2)}$, and $\theta^*_n \sim \tilde{\theta}_n$. We plug this estimate, Lemma~\ref{lemma_exponent}, and the expression of $C_d$ given in Proposition~\ref{prop_equivalent_monovariate} into (\ref{p_n_equivalent}), with $\tilde{\theta}_n = (\kappa_d/n)^{1/(d+1)}$; this gives $\alpha_d$ and $\beta_d$. The expansion of $F_d(n)$ is obtained by writing $\hat{\theta}_n = \tilde{\theta}_n(1 + \sum_{j \geq 1} e_j \tilde{\theta}_n^{2j})$ and solving $L_d'(\hat\theta_n) = -dn$ order by order. Since $p_{d,d-2} = 0$, the first correction is $e_1 = \frac{(d-2)\lambda_d}{d(d+1)\kappa_d}$ when $d \geq 4$, and $F_d(n) - Q_d(n^{1/(d+1)}) = -\frac{1}{2} f''(\tilde\theta_n)(\hat\theta_n - \tilde\theta_n)^2 + O(\tilde\theta_n^{6-d})$, which gives the leading term of $E_d$ and the value of $\gamma_4$.
\end{proof}

%%%%%%%%%%%%%%%%%%%%%%%%%%%%%%%%%%%%%%%%%%%%%%%%%%%

\section{Numerical considerations}\label{section_numerical_discu}

In this section, we make two numerical remarks, in order to clarify the behavior of our result when $n$ and $d$ grow large. 

\subsection{Discussion on \texorpdfstring{$I_{\text{crit}}(d,n)$}{I\_crit(d,n)}}
In general, without assuming that the non-trivial zeros of $\zeta$ are simple, $I_{\text{crit}}(d,n)$ can still be expressed as a sum over the zeros, but the term associated with a multiple zero is more complicated than in (\ref{I_crit_somme}). The first two zeros of the $\zeta$ function in the upper half of the critical strip are $r_1 = \frac{1}{2} + i\, 14.1347\ldots$ and $r_2 = \frac{1}{2} + i\, 21.0220\ldots$. Due to the exponential decay of the Gamma function along vertical lines, the term of (\ref{I_crit_somme}) associated with $r_1$ is about $10^4$ times larger than the term associated with $r_2$. \medskip 

The bounds on the density of the zeros given by Selberg \cite{Selberg:Riemann} suggest that the contribution of all the zeros other than $r_1$ and $\overline{r_1}$ to $I_{\text{crit}}(d,n)$ is negligible. Keeping only these two terms, we obtain the following numerical approximations in dimensions 2, 3, and 4:
\begin{equation}\label{numeric_I_crit}
\begin{aligned}[b]
    I_{\text{crit}}(2, n) &\approx n^{1/6}\Big(1.1310 {\scriptstyle\times} 10^{-10}\cos\left(4.7116\ln(0.6842\, n)\right)\\
    &\qquad - 1.7995{\scriptstyle\times}10^{-9}\sin\left(4.7116\ln(0.6842\,n)\right)\Big),\\
    I_{\text{crit}}(3, n) &\approx -n^{1/8}\Big(4.4775 {\scriptstyle\times} 10^{-10}\cos\left(3.5337\ln(0.2777\,n)\right)\\
    &\qquad + 5.3495 {\scriptstyle\times} 10^{-9}\sin\left(3.5337\ln(0.2777\,n)\right)\Big),\\
    I_{\text{crit}}(4, n) &\approx -n^{1/10}\Big(2.3568{\scriptstyle\times} 10^{-9}\cos\left(2.8269\ln(0.1305\,n)\right)\\
    &\qquad + 1.0669{\scriptstyle\times}10^{-8}\sin\left(2.8269\ln(0.1305\,n)\right)\Big).
\end{aligned}
\end{equation}
These approximations were obtained with \texttt{mpmath}; as a consistency check, we also computed $I_{\text{crit},d}(\theta)$ directly by numerical integration along the contour $\gamma$ (with $A = 0.1$), and obtained the same values as the sum of the residues, and as (\ref{numeric_I_crit}) up to the contribution of the zeros other than $r_1$, $\overline{r_1}$ (relative difference below $2\times 10^{-3}$ for $n = 10^6$ and $n = 10^{12}$, which is the order of the rounding of the coefficients).

To give an order of magnitude, $|I_{\text{crit}}(d,n)|$ is smaller than $4 \times 10^{-6}$ for all $n \leq 10^{20}$ and $d \in \{2,3,4\}$.\medskip

\subsection{About moving up in dimension}

For a rougher analysis of the logarithmic equivalent, that is, of the leading term of the exponential part, we can use the expansion $\zeta(d) = 1+ 2^{-d} + o(2^{-d})$ as $d $ grows large. We get
$$ \kappa_d^{\frac{1}{d+1}} = \left(\frac{2^{d-1}\zeta(d+1)}{\zeta(d)} \right)^{\frac{1}{d+1}} = 2 + O\left(\frac{1}{d}\right), \qquad\text{ as } d \rightarrow + \infty.$$
Therefore, in high dimension, the logarithm of $z_d(n)$ is close to $2(d+1)\, n^{\frac{d}{d+1}}$.\medskip

We wish to draw attention to the fact that \cite{Bureaux:polygons} and \cite{Barany:zonotopes} focus on polygonal lines (respectively zonotopes) starting at $\boldsymbol{0}$ and ending at a given point, whereas in this paper we enumerate the zonotopes inscribed in a hypercube. The hypercube $[-1, 1]^d$ can be split into $2^d$ hypercubes centered at the points $(\pm \frac{1}{2},\dots, \pm\frac{1}{2})$, and a zonotope inscribed in a hypercube can be seen as the Minkowski sum of zonotopes whose generators lie in each of the $2^{d-1}$ orthants with positive first coordinate. This short explanation, also given in \cite[Theorem 6.2]{Barany:zonotopes}, accounts for the additional factor $2^{d-1}$ in the leading term of the exponential part.

\subsection{Numerical validation}

The coefficients of $Zon_d$ can be computed exactly by expanding the product (\ref{partition_function}) truncated to the primitive vectors $\vv$ with $[\vv] \leq n\un$, which takes polynomial time in $n$ for fixed $d$. In dimension 2, the first values of $z_2(n)$, for $n = 0, 1, \dots, 12$, are
$$1,\ 3,\ 10,\ 34,\ 109,\ 331,\ 965,\ 2695,\ 7304,\ 19236,\ 49365,\ 123779,\ 303954.$$
Table~\ref{tab:validation} compares $z_2(n)$ with the estimate of Theorem~\ref{theorem_1}; the ratio tends to 1 slowly, its distance to 1 being roughly proportional to $n^{-1/3}$, which is consistent with relative corrections of order $\tilde{\theta}_n$.

\begin{table}[ht]
    \centering
    \begin{tabular}{|c|c|c|c|c|c|}
        \hline
        $n$ & 25 & 50 & 100 & 150 & 200 \\
        \hline
        $z_2(n) \big/ \alpha_2 n^{\beta_2} e^{Q_2(n^{1/3})}$ & 1.0613 & 1.0483 & 1.0365 & 1.0306 & 1.0269 \\
        \hline
    \end{tabular}
    \caption{Ratio between the exact number of zonotopes inscribed in $[0,n]^2$ and the estimate of Theorem~\ref{theorem_1} (the term $I_{\text{crit}}(2,n)$, smaller than $10^{-8}$ here, is neglected).}
    \label{tab:validation}
\end{table}

The same computation, keeping track of the number of generators, gives the exact mean and variance of the diameter. For $d=2$ and $n=120$, the ratio between the variance and the mean is $0.1915$, while Theorem~\ref{diameter_zono} predicts the limit $1 - \frac{1}{4} - \frac{2}{3\zeta(3)} = 0.1954\ldots$; for $d=3$ and $n=35$, the ratio is $0.1805$, for a predicted limit $1 - \frac{1}{8} - \frac{3}{4\zeta(4)} = 0.1820\ldots$. The constant $1-2^{-d}$ given by the unconditioned Boltzmann model (respectively $0.75$ and $0.875$) is thus clearly ruled out.

%%%%%%%%%%%%%%%%%%%%%%%%%%%%%%%%%%%%%%%%%%%%%%%%%%%

\section{Estimates of the moments of some parameters}\label{section_moments}

In this section, we establish the asymptotic behavior of the first two moments of two parameters of random lattice zonotopes. To do so, we add a variable $u$ (a marking variable, in the terminology of the symbolic method) which counts the parameter, and we study the coefficients of the resulting generating function for $u = e^{\eta}$, with $\eta$ real and close to $0$. This gives the moment generating function of the parameter, conditionally on the bounding box, and hence its moments.

We begin with a result similar to \cite[Lemma 4.3]{Barany:zonotopes}, which gives the mean number of generators of a zonotope contained in a given cone and ending at a given point, under the Boltzmann distribution. In the case of zonotopes inscribed in a hypercube, we compute the mean and the variance of the number of generators under the uniform distribution. The diameter of the graph of a zonotope is equal to its number of pairwise non-collinear generators \cite{Ziegler:lectures}, that is, to the size of the support of the associated function $\omega$. This gives Theorem~\ref{diameter_zono}.

\subsection{The tilted generating functions}

The generating function (\ref{partition_function}) has the structure of a partition function, and each factor $\left(1- \xx^{[\vv]}\right)^{-1} = \sum_{k \geq 0} \xx^{k[\vv]}$ accounts for the generator collinear to $\vv$: the $k$-th term corresponds to $\omega(\vv) = k$. Marking with the variable $u$ the number of generators of a zonotope turns each factor into $1 + u\sum_{k \geq 1}\xx^{k[\vv]}$. For $\eta \in \R$, we set
$$ Y_\eta(\xx) =  \prod_{ \vv\in \P_d^\star} \left( 1 + e^{\eta}\,\frac{\xx^{[\vv]}}{1 - \xx^{[\vv]}} \right), \qquad Y_\eta\left(e^{-\ttheta}\right) =  \prod_{ \vv\in \P_{d+}} \left( 1 + e^{\eta}\,\frac{e^{-\ttheta\cdot\vv}}{1 - e^{-\ttheta\cdot\vv}} \right)^{ 2^{d(\vv)-1}}, $$
so that $Y_0 = Zon_d$ and, denoting by $X_n$ the number of generators of a lattice zonotope drawn uniformly at random among those inscribed in $[0,n]^d$,
\begin{align}\label{mgf_generators}
\mathbb{E}\left[e^{\eta X_n}\right] = \frac{[\xx^{n\un}]\, Y_\eta(\xx)}{[\xx^{n\un}]\, Y_0(\xx)}.
\end{align}
Since $\log\left(1 + e^{\eta}\frac{q}{1-q}\right) = -\log(1-q) + \log\left(1-(1-e^{\eta})q\right) = \sum_{m \geq 1} \frac{1 - (1-e^{\eta})^m}{m}\, q^m$, the proof of Lemma~\ref{lemma_integral_form_gg} gives, for $|1-e^{\eta}| < 1$ and $c > d$,
\begin{align}\label{integral_tilted}
\log Y_\eta(e^{-\ttheta}) = \frac{1}{2 i \pi} \int _{c - i\infty} ^{ c+i \infty} \frac{D_\eta(s)}{\zeta(s)}\,\Gamma(s) F_d(s,\ttheta)\, ds, \qquad D_\eta(s) = \sum_{m \geq 1}\frac{1-(1-e^{\eta})^m}{m^{s+1}},
\end{align}
where $D_\eta$ is holomorphic on $\Re(s) > 0$ (and even for complex $\eta$ with $|1-e^\eta|<1$), and bounded on every half-plane $\Re(s) \geq c_0 > 0$, uniformly for $|1-e^{\eta}| \leq r < 1$. Note that $D_0(s) = \zeta(s+1)$.

\begin{lem}\label{lemma_uniform_tilted}
Let $\rho \in (0, \log 2)$. All the results of Sections~\ref{section_analysis} and~\ref{section_proof_main} on $Zon_d$ that are used in the proof of (\ref{p_n_equivalent}) hold for $Y_\eta$, uniformly for $\eta \in [-\rho,\rho]$, with $\zeta(d+1)$ replaced by $D_\eta(d)$, that is, $\kappa_d$ replaced by $\kappa_d(\eta) = 2^{d-1}D_\eta(d)/\zeta(d)$. In particular, as $n \rightarrow \infty$, uniformly for $\eta \in [-\rho,\rho]$,
$$[\xx^{n\un}]\, Y_\eta(\xx) \sim \frac{\exp\left(\ell_{d,\eta}(\theta^*_{n,\eta}) + dn\theta^*_{n,\eta}\right)}{\sqrt{(2\pi)^d\det \BB_\eta\left(e^{-\theta^*_{n,\eta}\un}\right)}},$$
where $\ell_{d,\eta}(\theta) = \log Y_\eta(e^{-\theta\un})$, $\BB_\eta$ is the matrix of second logarithmic derivatives of $Y_\eta$, and $\theta^*_{n,\eta}$ is the unique solution of $-\frac{1}{d}\ell_{d,\eta}'(\theta) = n$.
\end{lem}

\begin{proof}
We go through the proofs of Sections~\ref{section_analysis} and~\ref{section_proof_main}.

\emph{Lemma~\ref{lemma_derivativ}.} Its proof only uses the integral representation (\ref{equa_integral}), Lemma~\ref{lemma_Barnes}, and the boundedness of $\zeta(s+1)/\zeta(s)$ on the line $\Re(s) = c'$. By (\ref{integral_tilted}), the same proof applies to $Y_\eta$, with $\zeta(s+1)$ replaced by $D_\eta(s)$, uniformly in $\eta$.

\emph{Lemma~\ref{lemma_complex_points}.} Since $|1-(1-e^{\eta})^m| \leq 2$, the series expansion above gives $|\log Y_\eta(e^{-\boldsymbol{\tau}})| \leq 2\,\mathcal{L}_d(\Re(\boldsymbol{\tau}))$, and the proof of Lemma~\ref{lemma_complex_points} applies verbatim.

\emph{Conditions (I), (III), (IV).} They follow from the two previous points exactly as in Lemmas~\ref{condition1} and~\ref{condtion345}.

\emph{Conditions (II) and (V).} For $q \in (0,1)$ and $\eta \in \R$, the function
$$\varphi \mapsto \left(1 + e^{\eta}\frac{qe^{i\varphi}}{1-qe^{i\varphi}}\right)\Big/\left(1 + e^{\eta}\frac{q}{1-q}\right)$$
is the characteristic function of the probability distribution $(\pi_k)_{k \geq 0}$ on $\N$ given by $\pi_0 = \frac{1-q}{1-q+e^{\eta}q}$ and $\pi_k = \frac{e^{\eta}q^k(1-q)}{1-q+e^{\eta}q}$ for $k \geq 1$. For any probability distribution on $\N$ with characteristic function $\chi$,
$$1 - |\chi(\varphi)|^2 = \frac{1}{2}\sum_{k,l \geq 0}\pi_k\pi_l\left|e^{ik\varphi} - e^{il\varphi}\right|^2 \geq 4\,\pi_1\pi_2 \sin^2\left(\frac{\varphi}{2}\right),$$
and here $\pi_1\pi_2 \geq c_\rho\, q^3(1-q)^2$ for some $c_\rho>0$. Hence, with $q_{\vv} = e^{-\ttheta\cdot\vv}$,
$$\log \left(  \frac{|Y_\eta(e^{ - \ttheta + i \aalpha} )|}{Y_\eta(e^{ - \ttheta })} \right) \leq -2c_\rho \sum_{ \vv\in \P_{d+}} 2^{d(\vv)-1} q_{\vv}^3(1-q_{\vv})^2 \sin^2\left(\frac{\aalpha \cdot \vv}{2}\right).$$
This proves (V), and the second regime of the proof of Lemma~\ref{condition2} applies, with the family $\vv_i = (i,1,0,\dots,0)$ restricted to $M/2 \leq i \leq M$, for which $q_{\vv_i}$ and $1 - q_{\vv_i}$ are bounded from below (taking the constant $\varrho$ of that proof smaller, at least $M/4$ of these indices are good). The first regime only uses Taylor's formula, (I), and the previous points. This proves (II), uniformly in $\eta$.

\emph{Saddle point.} For each $\eta$, $-\frac{1}{d}\ell_{d,\eta}'$ is a decreasing bijection from $(0,+\infty)$ onto $(0,+\infty)$, which defines $\theta^*_{n,\eta}$. Since all the error terms in the proof of \cite[Theorem 4]{Gittenberger:hadmi} are controlled by the $o(1)$ terms in (I) and (II), which are uniform in $\eta$, the estimate of Proposition~\ref{definition_gittenberger} at $\ttheta = \theta^*_{n,\eta}\un$ holds uniformly in $\eta$.
\end{proof}

\subsection{Proof of Theorem~\ref{diameter_zono}}

\begin{proof}[Proof of Theorem~\ref{diameter_zono}]
Let $\Phi_n(\theta,\eta) = \ell_{d,\eta}(\theta) + dn\theta$ for $\theta>0$ and $\eta \in [-\rho,\rho]$, so that $\theta^*_{n,\eta}$ is the unique minimizer of $\Phi_n(\cdot,\eta)$, and set
$$\mathcal{F}_n(\eta) = \Phi_n(\theta^*_{n,\eta},\eta) - \frac{1}{2}\log\det\BB_\eta\left(e^{-\theta^*_{n,\eta}\un}\right).$$
By (\ref{mgf_generators}) and Lemma~\ref{lemma_uniform_tilted}, uniformly for $\eta \in [-\rho,\rho]$,
\begin{align}\label{mgf_asympt}
\log \mathbb{E}\left[e^{\eta X_n}\right] = \mathcal{F}_n(\eta) - \mathcal{F}_n(0) + o(1).
\end{align}
We now estimate the derivatives of $\mathcal{F}_n$; in what follows, $\theta = \theta^*_{n,\eta}$, and all the estimates are uniform in $\eta \in [-\rho,\rho]$. Write $c_{\vv} = 2^{d(\vv)-1}$, $q_{\vv} = e^{-\theta\,\un\cdot\vv}$, and $M_d = \frac{2^{d-1}}{\zeta(d)}$. By Proposition~\ref{proposition_eulerian} (applied with $\aalpha = \textbf{0}$ in each dimension $\delta \leq d$, and differentiated $j$ times with respect to $\theta$, which is allowed since the Mellin inversion integral converges absolutely together with its derivatives), for every fixed $j \geq 0$,
\begin{align}\label{moments_sums}
\sum_{\vv \in \P_{d+}} c_{\vv}\, (\un\cdot\vv)^j\, q_{\vv} \underset{\theta \rightarrow 0}{\sim} \frac{2^{d-1}}{\zeta(d)}\,\frac{(d+j-1)!}{(d-1)!}\,\theta^{-d-j},
\end{align}
and the same holds with $q_{\vv}$ replaced by $q_{\vv}^m$ (for fixed $m \geq 1$) if $\theta$ is replaced by $m\theta$ in the right-hand side. At $\eta = 0$, we have
\begin{align*}
\partial_\eta\Phi_n &= \sum_{\vv} c_{\vv}\, q_{\vv} \sim M_d\,\theta^{-d}, \\
\partial_\eta^2\Phi_n &= \sum_{\vv} c_{\vv}\,(q_{\vv} - q_{\vv}^2) \sim \left(1 - \frac{1}{2^d}\right)M_d\,\theta^{-d},\\
\partial_\theta\partial_\eta\Phi_n &= -\sum_{\vv} c_{\vv}\,(\un\cdot\vv)\,q_{\vv} \sim -d\,M_d\,\theta^{-d-1}, \\
\partial_\theta^2\Phi_n &= \ell_{d,0}''(\theta) \sim d(d+1)\,\kappa_d\,\theta^{-d-2},
\end{align*}
the last estimate following from Lemma~\ref{lemma_derivativ}. More generally, the expansion of $\log(1 + e^{\eta}\frac{q}{1-q})$ shows that, uniformly for $\eta \in [-\rho,\rho]$, $\partial_\theta^a\partial_\eta^b\Phi_n = O(\theta^{-d-a})$ for all $a \geq 0$ and $b \geq 1$, and $\partial_\theta^a\Phi_n = O(\theta^{-d-a})$ for $a \geq 2$, with $\partial_\theta^2\Phi_n \asymp \theta^{-d-2}$. By the implicit function theorem, $\frac{d\theta^*_{n,\eta}}{d\eta} = -\partial_\theta\partial_\eta\Phi_n/\partial_\theta^2\Phi_n = O(\theta)$, and the envelope theorem gives
$$\frac{d}{d\eta}\Phi_n(\theta^*_{n,\eta},\eta) = \partial_\eta\Phi_n, \qquad \frac{d^2}{d\eta^2}\Phi_n(\theta^*_{n,\eta},\eta) = \partial_\eta^2\Phi_n - \frac{(\partial_\theta\partial_\eta\Phi_n)^2}{\partial_\theta^2\Phi_n},$$
and differentiating once more, $\frac{d^3}{d\eta^3}\Phi_n(\theta^*_{n,\eta},\eta) = O(\theta^{-d})$. On the other hand, by Lemma~\ref{lemma_derivativ} for $Y_\eta$, which also holds, uniformly, for complex $\eta$ with $|\eta| \leq \rho$ by (\ref{integral_tilted}), $\log\det\BB_\eta(e^{-\theta\un}) = -d(d+2)\log\theta + h_n(\eta)$, where $h_n$ is bounded together with its first three derivatives, uniformly in $n$ (by Cauchy's estimates in the variable $\eta$); since $\frac{d}{d\eta}\log\theta^*_{n,\eta} = O(1)$, the derivatives of order $1$ to $3$ of $\frac{1}{2}\log\det\BB_\eta$ are $O(1)$. Altogether, with $\theta = \theta^*_{n,0}$ and using $\frac{M_d}{\kappa_d} = \frac{1}{\zeta(d+1)}$,
\begin{align*}
\mathcal{F}_n'(0) &= M_d\,\theta^{-d}(1+o(1)), \\
\mathcal{F}_n''(0) &= \left(1 - \frac{1}{2^d}\right)M_d\,\theta^{-d} - \frac{d^2M_d^2\,\theta^{-2d-2}}{d(d+1)\kappa_d\,\theta^{-d-2}} + o(\theta^{-d}) \\
&= \left(1 - \frac{1}{2^{d}} - \frac{d}{(d+1)\,\zeta(d+1)}\right)M_d\,\theta^{-d}\,(1+o(1)),
\end{align*}
and $\mathcal{F}_n'''(\eta) = O(\theta^{-d})$ uniformly for $\eta \in [-\rho,\rho]$. The constant $\sigma_d^2 = 1 - 2^{-d} - \frac{d}{(d+1)\zeta(d+1)}$ is positive for all $d \geq 2$: indeed $\zeta(d+1) > 1$, so that $\sigma_d^2 > \frac{1}{d+1} - \frac{1}{2^d} > 0$.

Let $s_n^2 = \mathcal{F}_n''(0) \asymp \theta^{-d}$ and $t \in \R$. Applying (\ref{mgf_asympt}) with $\eta = t/s_n \rightarrow 0$ and Taylor's formula,
$$\log \mathbb{E}\left[e^{t(X_n - \mathcal{F}_n'(0))/s_n}\right] = \frac{t^2}{2} + O\left(\frac{|t|^3\,\theta^{-d}}{s_n^3}\right) + o(1) \underset{n \rightarrow \infty}{\longrightarrow} \frac{t^2}{2}.$$
Hence $(X_n - \mathcal{F}_n'(0))/s_n$ converges in distribution to the standard normal distribution, and, since its moment generating function converges on a neighbourhood of $0$, all its moments converge to those of the standard normal distribution. In particular, $\mathbb{E}[X_n] = \mathcal{F}_n'(0) + o(s_n) \sim M_d\,\theta^{-d}$ and $\mathrm{Var}(X_n) \sim s_n^2$. Finally, $\theta = \theta^*_{n,0} \sim \tilde{\theta}_n = (\kappa_d/n)^{1/(d+1)}$, and $M_d\,\tilde{\theta}_n^{-d} = \frac{\kappa_d^{1/(d+1)}}{\zeta(d+1)}\, n^{\frac{d}{d+1}}$. Since the diameter of the graph of a zonotope is its number of generators, this proves Theorem~\ref{diameter_zono}.
\end{proof}

The proof shows in addition that the diameter satisfies a central limit theorem: $(X_n - \mu^n_{\text{diam}})/\sigma^n_{\text{diam}}$ converges in distribution to the standard normal distribution.

\begin{rem}
The first two terms of the constant $\sigma_d^2$ correspond to the variance of the number of generators under the (unconditioned) Boltzmann distribution, while the third one accounts for the conditioning on the bounding box. For $d=1$, the same computation applied to integer partitions gives the constant $\frac{1}{2} - \frac{3}{\pi^2}$ for the variance of the number of distinct part sizes, in agreement with the result of Goh and Schmutz \cite{GohSchmutz}. Numerical values for $d \in \{2,3\}$ are given in Section~\ref{section_numerical_discu}.
\end{rem}

\subsection{Number of occurrences of a generator}

\begin{prop}\label{occurences}
Let $\vv_0$ be a primitive vector, and let $W_n$ be the number of occurrences $\omega(\vv_0)$ of $\vv_0$ in a lattice zonotope drawn uniformly at random among those inscribed in $[0,n]^d$. Then $\kappa_d^{1/(d+1)}\|\vv_0\|_1\, n^{-1/(d+1)}\,W_n$ converges in distribution, with convergence of all moments, to the exponential distribution of parameter 1. In particular, its mean $\mu_{\text{occ}}^n$ and its variance $(\sigma^2)^n_{\text{occ}}$ satisfy
$$\mu_{\text{occ}}^n \underset{n \rightarrow + \infty}{\sim} \frac{n^{\frac{1}{d+1}}}{\kappa_d^{\frac{1}{d+1}}\, \|\vv_0\|_1}, \qquad (\sigma^2)^n_{\text{occ}}  \underset{n \rightarrow + \infty}{\sim}\left(\frac{n^{\frac{1}{d+1}}}{\kappa_d^{\frac{1}{d+1}}\, \| \vv_0\|_1} \right)^2. $$
\end{prop}

\begin{proof}
By symmetry, we can assume that $\vv_0 \in \P_{d+}$. Marking the occurrences of $\vv_0$ with $u = e^\eta$ replaces the factor $(1 - \xx^{\vv_0})^{-1}$ by $(1 - e^{\eta}\xx^{\vv_0})^{-1}$; let
$$Z_\eta(\xx) = Zon_d(\xx)\,\frac{1 - \xx^{\vv_0}}{1 - e^{\eta}\,\xx^{\vv_0}}, \qquad \mathbb{E}\left[e^{\eta W_n}\right] = \frac{[\xx^{n\un}]\, Z_\eta(\xx)}{[\xx^{n\un}]\, Zon_d(\xx)}.$$
Fix $t \in (-1,1)$ and let $\eta = t\,\theta\,\|\vv_0\|_1$ with $\theta = \theta^*_{n,0}$, so that $y = e^{-\theta\|\vv_0\|_1}$ satisfies $e^{\eta}y \leq e^{-(1-|t|)\theta\|\vv_0\|_1}$. The additional factor $\frac{1-y}{1-e^{\eta}y}$ and its logarithmic derivatives are of lower order: its contributions to $\ba$, $\BB$, and $\CC$ are $O(\theta^{-1})$, $O(\theta^{-2})$, and $O(\theta^{-3})$; for complex arguments its modulus is at most $\frac{|1-\xx^{\vv_0}|}{1 - e^{\eta}y} \leq \frac{2}{1-e^\eta y} = O(\theta^{-1})$ times its value at $\xx = e^{-\theta\un}$, which is absorbed by the super-polynomial decay in (II). Hence, as in Lemma~\ref{lemma_uniform_tilted}, the estimate of Proposition~\ref{definition_gittenberger} holds for $Z_\eta$, and, as in the previous proof, $\log\mathbb{E}[e^{\eta W_n}] = \mathcal{G}_n(\eta) - \mathcal{G}_n(0) + o(1)$, where $\mathcal{G}_n$ is defined as $\mathcal{F}_n$ with $Y_\eta$ replaced by $Z_\eta$. The envelope theorem gives
$$\mathcal{G}_n(\eta) - \mathcal{G}_n(0) = \log\frac{1-y}{1-e^{\eta}y} + O\left(\eta^2\,\theta^{-4}\,\theta^{d+2}\right) + O(\eta) = \log\frac{1}{1-t} + o(1),$$
since the coupling term $\partial_\theta\partial_\eta\Phi_n$ is now of order $\theta^{-2}$. Therefore $\mathbb{E}[e^{t\,\theta\,\|\vv_0\|_1 W_n}] \rightarrow \frac{1}{1-t}$ for all $t \in (-1,1)$, which is the moment generating function of the exponential distribution of parameter 1. This proves the convergence in distribution and of all moments, and the proposition follows from $\theta = \theta^*_{n,0} \sim (\kappa_d/n)^{1/(d+1)}$.
\end{proof}

%%%%%%%%%%%%%%%%%%%%%%%%%%%%%%%%%%%%%%%%%%%%%%%%%%%

\section*{Acknowledgements}
The authors wish to thank Philippe Marchal for his support and for his valuable comments, remarks, and corrections on the first version of this paper.

\end{document}